\documentclass[a4paper,11pt,leqno]{amsart}
\usepackage{amsmath,amscd,amsfonts,amssymb}
\usepackage{latexsym}
\usepackage{graphics}
\usepackage{epsfig}
\usepackage[french,english]{babel}

\numberwithin{equation}{section}
\newtheorem{theorem}{Th\'eor\`eme}[section]
\newtheorem{prop}[theorem]{Proposition}
\newtheorem{cor}[theorem]{Corollaire}
\newtheorem{lemma}[theorem]{Lemme}

\DeclareMathOperator{\tr}{tr}

\def\into{\hookrightarrow}
\newcommand{\Mbar}{\overline{M}}
\newcommand{\gbar}{\overline{g}}
\newcommand{\nabbar}{\overline{\nabla}}
\def\binf{\partial_\infty M}
\def\bint{\partial_0 M}
\newcommand{\bR}{\mathbb{R}}
\newcommand{\phitil}{\widetilde{\varphi}}
\newcommand{\rlie}{\mathring{\mathcal{L}}}
\DeclareMathOperator{\im}{Im}
\DeclareMathOperator{\divg}{div}

\newcommand{\scal}{\mathrm{Scal}}
\newcommand{\hscal}{\widehat{\scal}}
\newcommand{\ric}{\mathrm{Ric}}

\newcommand{\riemuddd}[4]{R^{#1}_{\phantom{#1} #2 #3 #4}}
\newcommand{\ricud}[2]{\ric^{#1}_{\phantom{#1} #2}}
\newcommand{\ricuu}[2]{\ric^{#1 #2}}
\newcommand{\hess}{\mathrm{Hess}}

\title[Courbure scalaire prescrite et contraintes sur une vari\'et\'e AH]{De l'\'equation de prescription de courbure scalaire aux \'equations de contrainte en relativit\'e g\'en\'erale sur une vari\'et\'e asymptotiquement hyperbolique}

\author{Romain Gicquaud}
\address{Adresse actuelle : \newline
Institut de Math\'ematiques et de Mod\'elisation de Montpellier \newline
UMR 5149 CNRS - Universit\'e Montpellier II \newline
Case Courrier 051 - Place Eug\`ene Bataillon \newline
34095 Montpellier, France \newline }
\email{Romain.Gicquaud@math.univ-montp2.fr}

\graphicspath{{images/}}

\begin{document}
\date{}

\maketitle

\selectlanguage{francais}
\begin{abstract}
Cet article est consacr\'e \`a l'\'etude de deux probl\`emes sur les vari\'et\'es asymptotiquement hyperboliques avec un bord interne. Dans un premier temps, nous examinons le probl\`eme de la prescription de la courbure scalaire pour des conditions au bord interne de Dirichlet et de prescription de la courbure moyenne. Nous appliquons ensuite les r\'esultats obtenus \`a l'\'etude de l'\'equation de Lichnerowicz avec une condition au bord interne d'horizon apparent (pass\'e ou futur). La derni\`ere partie traite de la construction de TT-tenseurs. On construit ainsi des donn\'ees initiales, \`a courbure moyenne constante, pour  les \'equations d'Einstein du vide.
\end{abstract}

\selectlanguage{english}
\begin{abstract}
Two problems concerning asymptotically hyperbolic manifolds with an inner boundary are studied. First, we study scalar curvature presciption with either Dirichlet or mean curvature prescription interior boundary condition. Then we apply those results to the Lichnerowicz equation with (future or past) apparent horizon interior boundary condition. In the last part we show how to construct TT-tensors. Thus we obtain Cauchy data with constant mean curvature for Einstein vacuum equations.
\end{abstract}
\medskip

\selectlanguage{french}
\noindent {\bf Mots Clefs} : Vari\'et\'es asymptotiquement hyperboliques,
courbure scalaire conforme, \'equation de Lichnerowicz, TT-tenseurs,
\'equations  de contraintes, horizons apparents.
\\
\newline

\section{Introduction}

Dans cet article, nous \'etudions deux probl\`emes : celui de la prescription de la courbure scalaire sur une vari\'et\'e asymptotiquement hyperbolique et la construction de solutions des \'equations de contrainte en relativit\'e g\'en\'erale contenant des horizons apparents.\\

Dans la premi\`ere partie, nous rappelons les r\'esultats de \cite{GrahamLee} et \cite{LeeFredholm}, d\'emontrons des variantes du principe du maximum g\'en\'eralis\'e tel qu'il appara\^it dans \cite{GrahamLee} et finalement nous \'etendons la m\'ethode de monotonie au cas de sur et sous-solutions au sens des distributions. Ces r\'esultats nous seront utiles pour r\'esoudre les \'equations de prescription de la courbure scalaire et de Lichnerowicz.\\

Pour le probl\`eme de la prescription de la courbure scalaire, nous g\'en\'eralisons les r\'esultats de \cite{DelayCourbScal} au cas des vari\'et\'es asymptotiquement hyperboliques avec un bord interne (\'eventuellement vide). Rappelons que pour une m\'etrique riemannienne $g$ de courbure scalaire $\scal$, si $\varphi$ est une fonction positive alors la courbure scalaire $\hscal$ de $\hat{g} = \varphi^\kappa g$, avec $\kappa=\frac{4}{n-2}$ est donn\'ee par (voir par exemple \cite{Besse}) :

$$\hscal = \varphi^{-\kappa-1} \left( - \frac{4(n-1)}{n-2} \Delta \varphi + \scal~\varphi\right).$$

Par convention, le laplacien est donn\'e par $\Delta f = g^{ij}\left(\partial_i \partial_j f - \Gamma^k_{ij} \partial_k f\right)$. Nos r\'esultats principaux sont les th\'eor\`emes \ref{PrescScalDiri} et \ref{PrescScalH} :

\begin{theorem} Soit $(M^n, g)$ une vari\'et\'e asymptotiquement hyperbolique de classe $\mathcal{C}^{l, \beta}$. Soient deux fonctions $\scal,~\hscal \in \mathcal{C}^{k-2, \alpha}_0(M)$ ($2 < k+\alpha \leq l + \beta$, $0 < \alpha < 1$) avec $\hscal < 0$ telles que $\scal,~\hscal \to_{\binf} -n(n-1)$.

\begin{enumerate}
\item \textbf{Probl\`eme de Dirichlet le long de $\bint$} (Th\'eor\`eme \ref{PrescScalDiri}): Soit $\varphi_0 > 0$ une fonction $\mathcal{C}^{k, \alpha}$ sur $\partial_0 M$ alors il existe une solution unique $\varphi \in~\mathcal{C}^{k, \alpha}_0$ telle que $\varphi > 0$ et $\varphi \to_{\binf} 1$, au probl\`eme de Dirichlet :

\begin{equation}
\label{EqPrescScalDiriIntro}
\left\lbrace
\begin{array}{rcl}
-\frac{4 (n-1)}{n-2} \Delta \varphi + \scal~\varphi & = & \hscal~\varphi^{\kappa+1}\\
\varphi & = & \varphi_0\quad\mathrm{sur}~\bint.\\
\end{array}
\right.
\end{equation}

\item \textbf{Prescription de la courbure moyenne de $\bint$} (Th\'eor\`eme \ref{PrescScalH}): Soient $H, \widehat{H} \in \mathcal{C}^{k-1, \alpha}\left(\bint\right)$, avec $\widehat{H} \geq 0$, il existe une unique solution $\varphi \in \mathcal{C}^{k, \alpha}_0$ telle que $\varphi > 0$ et $\varphi \to_{\binf} 1$ au probl\`eme de la prescription de la courbure moyenne de $\bint$ :

\begin{equation}
\label{CourbMoyIntro}
\left\{
\begin{array}{l}
-\frac{4(n-1)}{n-2} \Delta \varphi + \scal~\varphi = \hscal~\varphi^{\kappa+1}\\
\frac{2}{n-1} \nabla_\nu \varphi - H \varphi = - \widehat{H} \varphi^{\frac{\kappa}{2}+1} \text{ le long de }\bint,
\end{array}
\right.
\end{equation}
\end{enumerate}
o\`u $\nu$ d\'esigne la normale \`a $\bint$ sortant de $M$. De plus si $\scal-\hscal \in \mathcal{C}^{k-2, \alpha}_\delta$ pour un certain $\delta \in [0; n[$ alors $\varphi-1 \in \mathcal{C}^{k, \alpha}_\delta$.
\end{theorem}

Signalons que le probl\`eme de la prescription de la courbure scalaire sur une vari\'et\'e \`a courbure n\'egativement pinc\'ee a \'et\'e trait\'e dans \cite{AvilesMcOwen}, ce cadre est cependant trop g\'en\'eral pour esp\'erer un contr\^ole de la fonction $\varphi$ \` a l'infini. Pour des valeurs $\delta \geq n$, la fonction $\varphi$ solution de l'\'equation de prescription de la courbure scalaire n'a plus n\'ecessairement un comportement asymptotique aussi fort que celui de $\hscal - \scal$, une obstruction importante est alors fournie par la masse \cite{Wang, ChruscielHerzlich} (voir \'egalement \cite{MinOo}).\\

Dans une seconde partie, nous montrons l'existence de solutions \`a l'\'equation de Lichnerowicz avec la condition au bord d'horizon apparent (pass\'e et futur). Rappelons que l'\'equation de Lichnerowicz (correspondant \`a la premi\`ere ligne de \ref{EqLichHorAppIntro}) appara\^it naturellement dans la construction de donn\'ees initiales \`a courbure moyenne constante pour les \'equations d'Einstein du vide (voir section 4) :

\begin{theorem}[Construction de solutions de l'\'equation de Lichnerowicz contenant des horizons apparents] Soient $M$ une vari\'et\'e asymptotiquement hyperbolique de classe $\mathcal{C}^{l, \beta}$ avec $l+\beta \geq 2$ et $\tau \in [-1; 1]$. Fixons, pour chaque composante connexe $\sigma_i$ de $\bint$, un r\'eel $\epsilon_i = \pm 1$ ($\epsilon_i = -1$ pour un horizon apparent futur et $\epsilon_i = +1$ pour un horizon apparent pass\'e). Supposons de plus, lorsque $\tau = \pm 1$, que, pour tout $i$ tel que $\epsilon_i = -\tau$, l'invariant de Yamabe $\mathcal{Y}(\sigma_i) > 0$. Soit $L \in \mathcal{C}^{k-1, \alpha}_0(M, T^{*2}M)$ un 2-tenseur sym\'etrique de trace nulle tel que $|L|^2_g \to 0$ au voisinage de $\binf$ et tel que $\epsilon_i L_{\nu_i \nu_i} \geq 0$ sur $\bint$. Il existe une solution $\varphi > 0$ au probl\`eme :

\begin{equation}
\label{EqLichHorAppIntro}
\left\lbrace
\begin{aligned}
-\frac{4 (n-1)}{n-2} \Delta \varphi + \scal~ \varphi + n(n-1) \varphi^{\kappa+1} - |L|_g^2 \varphi^{-\kappa-3} = 0 & \quad\text{sur $\mathring{M}$}\\
\frac{2(n-1)}{n-2}\nabla_{\nu_i}\varphi - H_i \varphi = \epsilon_i \left[L_{\nu_i\nu_i}\varphi^{-1-\frac{\alpha}{2}} - (n-1)\tau~\varphi^{\frac{\alpha}{2}+1}\right] & \quad\text{sur $\sigma_i$ pour tout $i$},
\end{aligned}
\right.
\end{equation}
o\`u $\scal$ est le scalaire de courbure de $(M, g)$ et $\nu_i$ la normale \`a $\sigma_i$ sortante (i.e. dirig\'ee vers l'int\'erieur de l'horizon apparent) avec $\varphi \to 1$ au voisinage de $\binf$. De plus si $\scal + n(n-1) \in \mathcal{C}^{k-2, \alpha}_\delta$ et $|L|^2 \in \mathcal{C}^{k-2, \alpha}_\delta$ avec $\delta \in [0; n)$, alors $\varphi-1\in \mathcal{C}^{k, \alpha}_\delta$.
\end{theorem}

Remarquons qu'avec nos conventions le cas $\epsilon_i \tau = -1$ correspond \`a un horizon apparent futur (resp. pass\'e) dans les hypersurfaces asymptotiquement isotropes qui sont des surfaces de Cauchy pour le d\'eveloppement en temps futur (resp. pass\'e). En page \pageref{secCtrEx} nous montrons que si l'invariant de Yamabe du bord n'est pas strictement positif, il peut ne pas exister de solution \`a l'\'equation \eqref{EqLichHorAppIntro}.\\

Notons que le probl\`eme sans bord a d\'ej\`a \'et\'e \'etudi\'e dans \cite{AnderssonChrusciel}. Les horizons apparents sont la version \og donn\'ees initiales\fg~ des horizons des \'ev\`enements (voir par exemple \cite{Wald}). Cette construction est la version asymptotiquement hyperbolique des articles \cite{DainInitialData}, \cite{DainTrapped1,DainTrapped2}, \cite{Maxwell}. Le cas des vari\'et\'es asymptotiquement hyperboliques permet en particulier de construire des surfaces de Cauchy asymptotiquement isotropes dans une vari\'et\'e lorentzienne asymptotiquement plate et des surfaces de Cauchy correspondant aux sections \`a courbure moyenne nulle des espaces asymptotiquement anti-de Sitter. Signalons \'egalement que d'autres auteurs ont construit des m\'etriques contenant des horizons apparents par des techniques de recollement \cite{ShiTam}, \cite{LiShiWu}.\\

Li\'e \`a l'\'equation de Lichnerowicz, nous \'etudions finalement la construction de TT-tenseurs. Notre construction, bas\'ee sur \cite{LeeFredholm},  permet d'obtenir l'intervalle des poids optimal pour l'existence de TT-tenseurs. Ceci nous permet de construire des solutions des \'equations de contrainte en relativit\'e g\'en\'erale :

\begin{theorem} Soit $(M, g)$ une vari\'et\'e asymptotiquement hyperbolique de classe $\mathcal{C}^{l, \beta}$ avec $l+\beta \geq 2$.

\begin{enumerate}
\item Soient $L_0$ un 2-tenseur sym\'etrique et sans trace avec $L_0 \in W^{k-1, p}_\delta$ avec $2 \leq k \leq l$, $1 < p < \infty$ et $\left|\delta + \frac{n-1}{p} - \frac{n-1}{2} \right| < \frac{n+1}{2}$, et $b \in W^{k-1-\frac{1}{p}, p} (\bint, T^*M)$. Il existe une unique 1-forme $\psi \in W^{k, p}_\delta(M, T^*M)$ telle que, si on d\'efinit
$L = L_0 + \mathcal{L}_{\psi^\sharp} g - \frac{2}{n} \mathrm{div}\left(\psi^\sharp\right) g$, alors :

\begin{itemize}
\item $L$ est un tenseur sym\'etrique transverse et sans trace
\item $L(\nu, .) = b$ sur $\bint$.
\end{itemize}

\item Soient $L_0$ un 2-tenseur sym\'etrique et sans trace avec $L_0 \in \mathcal{C}^{k-1, \alpha}_\delta$ avec $2 \leq k + \alpha \leq l$,
$0 < \alpha < 1$ et $\left|\delta - \frac{n-1}{2} \right| < \frac{n+1}{2}$, et $b \in \mathcal{C}^{k-1, \alpha} (\bint, T^*M)$. Il existe une unique 1-forme $\psi \in \mathcal{C}^{k, \alpha}_\delta(M, T^*M)$ telle que, si on d\'efinit
$L = L_0 + \mathcal{L}_{\psi^\sharp} g - \frac{2}{n} \mathrm{div}\left(\psi^\sharp\right) g$, alors :

\begin{itemize}
\item $L$ est un tenseur sym\'etrique transverse et sans trace
\item $L(\nu, .) = b$ sur $\bint$.
\end{itemize}
\end{enumerate}
\end{theorem}

Je remercie Erwann Delay, mon directeur de th\`ese, pour ses nombreux conseils et sa relecture attentive des versions pr\'eliminaires de cet article. Je remercie \'egalement Piotr Chru\'sciel et Marc Herzlich pour leurs commentaires qui m'ont guid\'es tout au long de ce travail. Enfin, je suis reconnaissant envers Beno\^it Michel pour son aide sur le contre-exemple de la section \ref{secCMC}.

\section{Pr\'eliminaires}

L'objectif de cette section est de rappeler les r\'esultats principaux de \cite{GrahamLee} et \cite{LeeFredholm} ainsi que d'\'etendre la m\'ethode de monotonie au cas de sur et sous-solutions au sens des distributions pour des conditions au bord non lin\'eaires.

\subsection{Vari\'et\'es asymptotiquement hyperboliques}

Soit $\overline{M}^n$ une vari\'et\'e compacte \`a bord. On suppose que $\partial M$ est s\'epar\'e en deux parties ouvertes $\binf$ et $\bint$ (r\'eunion de composantes connexes). On notera $M = \mathring{M} \bigcup \bint$. On appelle \textbf{fonction d\'efinissante} pour $\binf$ une fonction lisse $\rho : \overline{M} \to [0;~\infty[$ telle que $\rho^{-1}(0) = \binf$ et $d\rho \neq 0$ le long de $\binf$. Une m\'etrique $g$ sur $M$ sera dite \textbf{conform\'ement compacte} de classe $\mathcal{C}^{l, \beta}$  si $\rho^2 g$ s'\'etend en une m\'etrique $\overline{g} \in \mathcal{C}^{l, \beta}$ sur $\overline{M}$. Un calcul simple montre alors que, si $g$ est conform\'ement compacte de classe $\mathcal{C}^{l, \beta}$ avec $l + \beta \geq 2$, la courbure sectionnelle de $g$ tend vers $-|d\rho|^2_{\overline{g}}$ au voisinage de $\binf$. On dira donc que $g$ est \textbf{asymptotiquement hyperbolique} si $g$ est conform\'ement compacte avec $|d\rho|^2_{\overline{g}} = 1$ le long de $\binf$. On notera par la suite $M_\sigma = \rho^{-1}(]0; \sigma[)$ et $S_\sigma = \rho^{-1}(\sigma)$. $S_\sigma$ est une sous-vari\'et\'e si $\sigma$ est assez petit.

\subsection{Espaces de fonctions}

On fixe un fibr\'e vectoriel g\'eom\'etrique (i.e. associ\'e au fibr\'e principal $SO(M)$) E sur $M$. Par la suite, on se fixe $k \geq 0$ un entier et $0 < \alpha < 1$. On d\'efinit tout d'abord l'espace $\mathcal{C}^{k, \alpha}_{(0)}$. C'est l'ensemble des sections $f : \overline{M} \to \overline{E}$ de classe $\mathcal{C}^{k, \alpha}$ qu'on munit de la norme standard.\\

On d\'efinit ensuite l'espace de Sobolev $W^{k, p}_0(M, E)$. C'est l'espace de Sobolev des sections $u \in L^p$ telles qu'au sens des distributions, $\forall~j\in \{0, \cdots, k\},~\nabla^{(j)} u \in L^p$. On le munit de la norme :

$$\|u\|_{W^{k,p}_0(M, E)} = \left( \sum_{j=0}^k \int_M \left|\nabla^{(j)} u\right|^p d\mu_g \right)^{\frac{1}{p}}.$$

Puis l'espace de Sobolev \`a poids $W^{k, p}_\delta (M, E) = \rho^\delta W^{k, p}_0 (M, E)$ qu'on munit de la norme
$\|u\|_{W^{k,p}_\delta(M, E)} = \|\rho^{-\delta} u\|_{W^{k,p}_0(M, E)}$. On notera $L^p_\delta = W^{0, p}_\delta$.\\

Finalement nous d\'efinissons les espaces de H\"older \`a poids. On choisit tout d'abord un nombre fini de cartes $\phi = (\rho, \theta^1, \cdots, \theta^{n-1})$ au voisinage de $\binf$. telles que le domaine de d\'efinition des $\phi$ recouvre $\binf$ qu'on compl\`ete avec un nombre fini de cartes dont le domaine de d\'efinition est pr\'ecompact dans $M$. On note $\mathbb{H}$ l'espace hyperbolique vu comme le demi-espace de $\bR^n$ $\{ x_1 > 0 \}$ muni de la m\'etrique $g_{hyp} = \frac{1}{x_1^2} g_{eucl}$. et on d\'efinit $B_r$ la boule centr\'ee en $(1,0, \cdots, 0)$ de rayon $r$ dans $\mathbb{H}$ pour la m\'etrique hyperbolique. Si $M \ni p_0 = \phi^{-1}(\rho_0, \theta^1_0\cdots, \theta^{n-1}_0)$ est l'image r\'eciproque par une des cartes fix\'ees, on d\'efinit les coordonn\'ees de M\"obius au voisinage en $p_0$ par :

$$\phi^r_{p_0}\left(p\right) = \left(\frac{\rho(p)}{\rho_0}, \frac{\theta^1(p)-\theta^1_0}{\rho_0}, \cdots, \frac{\theta^{n-1}(p)-\theta^{n-1}_0}{\rho_0} \right)$$

$\phi^r_{p_0} : \left(\phi^r_{p_0}\right)^{-1}\left(B_r\right) \to B_r$. On introduit ensuite la norme de H\"older :

$$\|u\|_{\mathcal{C}^{k, \alpha}_\delta(M, E)} = \sup_{p_0 \in M} \rho^{-\delta}(p_0) \left\| \left(\left(\phi^1_{p_0}\right)^{-1}\right)^* u\right\|_{\mathcal{C}^{k, \alpha}(B_1)}.$$

L'espace $\mathcal{C}^{k, \alpha}_\delta(M, E)$ est alors l'espace des sections $u \in \mathcal{C}^{k, \alpha}_{\mathrm{loc}}$ telles que :

$$\|u\|_{\mathcal{C}^{k, \alpha}_\delta(M, E)} < \infty.$$

\subsection{La m\'ethode de monotonie}\label{secMonotonie}

Pour r\'esoudre l'\'equation de prescription de la courbure scalaire, nous utilisons la m\'ethode de monotonie adapt\'ee pour prendre en compte des conditions au bord non lin\'eaires telle qu'elle est d\'ecrite dans \cite{Maxwell}. Cette m\'ethode nous sera \'egalement utile par la suite pour la r\'esolution de l'\'equation de Lichnerowicz (voir section \ref{secCMC}). Cependant, contrairement \`a \cite{Maxwell}, nous regardons l'existence de solutions dans les espaces de H\"older \`a poids. Nos sur et sous-solutions seront des fonctions lipschitziennes, d\'erivables dans un voisinage du bord interne $\bint$, sur et sous-solutions au sens des distributions. Dans tout ce paragraphe, nous fixons une vari\'et\'e $(M^n, g)$ asymptotiquement hyperbolique de classe $\mathcal{C}^{l, \beta}$ ($l+\beta \geq 2$), un entier $k \geq 2$ et $\alpha$ un r\'eel, $0 < \alpha < 1$ tels que
$2 < k+\alpha \leq l + \beta$. Supposons qu'on souhaite r\'esoudre le probl\`eme suivant :

\begin{equation}
\label{edpgen}
- \Delta \varphi = F(p, \varphi)
\end{equation}

\noindent o\`u $F : M \times \bR_+^* \to \bR$ est une fonction de la forme $F(p, \varphi) = \sum_{i \in I} a_i(p) \varphi^{\beta_i}$ avec $I$ est un ensemble fini, $\beta_i \in \bR$ et $a_i \in \mathcal{C}^{k, \alpha}_0(M)$. Notons que si $\forall~i\in I, \beta_i = 0 \textrm{ ou } \beta_i \geq 1$, $F$ se prolonge en une fonction d\'erivable \`a $M \times \bR_+$, ce qui sera sous-entendu par la suite.

\begin{prop}\label{propMonotonie}
On suppose qu'il existe deux fonctions $\varphi_+$ et $\varphi_-$ appel\'ees respectivement \textbf{sur-solution} et \textbf{sous-solution} \`a valeurs dans $(\epsilon; \infty)$ ~pour un certain $\epsilon > 0$ petit (ou dans $\bR_+$ dans le cas o\`u les $\beta_i$ sont nuls ou sup\'erieurs ou \'egaux \`a $1$), continues sur $\Mbar$, $\varphi_+ \geq \varphi_-$, telles que :

\begin{eqnarray*}
\forall~\psi \in \mathcal{C}^{2}_c\left( \mathring{M} \right) & -\int_M \varphi_+ \Delta \psi \geq \int_M \psi F\left(p, \varphi_+\right) \\
 & \left(\text{resp.~}-\int_M \varphi_- \Delta \psi \leq \int_M \psi F\left(p, \varphi_-\right) \right)
\end{eqnarray*}
alors, si on sait que (estimation a priori) si $\varphi$ est solution avec $\varphi_- \leq \varphi \leq \varphi_+$, alors
$\varphi \geq \epsilon > 0$ pour un certain $\epsilon$, on a :

\begin{enumerate}
\item \textbf{Probl\`eme de Dirichlet le long de $\bint$}

Soit $h \in \mathcal{C}^{k, \alpha}\left(\bint\right)$. Si $\varphi_- \leq h \leq \varphi_+$ sur $\bint$, il existe une fonction
$\varphi \in \mathcal{C}^{k, \alpha}_0$, $\varphi_- \leq \varphi \leq \varphi_+$, solution du probl\`eme de Dirichlet :

\begin{equation}
\label{dirichlet}
\left\lbrace
\begin{array}{rcl}
- \Delta \varphi & = & F(p, \varphi)\\
\varphi & = & h \quad\mathrm{sur~}\bint.
\end{array}
\right.
\end{equation}

\item \textbf{Condition au bord $\bint$ non lin\'eaire}

On se donne ici une fonction $f : \bint \times \bR_+^* \to \bR$ de classe $\mathcal{C}^{k-1, \alpha}$. On suppose de plus que $\partial_\nu \varphi_+ \geq f\left(p, \varphi_+\right)$ (resp. $\partial_\nu \varphi_- \leq f\left(p, \varphi_-\right)$) $\forall p \in \bint$. Alors il existe une fonction $\varphi \in \mathcal{C}^{k, \alpha}_0$ telle que
$\varphi_- \leq \varphi \leq \varphi_+$ solution de :

\begin{equation}
%\label{nonlin}
\left\lbrace
\begin{array}{rcl}
- \Delta \varphi & = & F(p, \varphi)\\
\partial_\nu \varphi & = & f(p, \varphi) \quad\mathrm{sur~}\bint.
\end{array}
\right.
\end{equation}
\end{enumerate}
\end{prop}

D\'emontrons le second point, le premier \'etant plus simple. Notre preuve est bas\'ee sur des lemmes issus de \cite{GrahamLee}. Commen\c cons par le lemme suivant :

\begin{lemma}[Principe du maximum g\'en\'eralis\'e]\label{MaxYau} Soient $M$ une vari\'et\'e asymptotiquement hyperbolique et $f \in \mathcal{C}^2(M)$ une fonction major\'ee. On suppose que $f$ n'atteint pas son maximum en un point de $\bint$. Alors il existe une suite $p_i \in M$ telle que :

\begin{enumerate}
\item $\lim_{i \to \infty} f\left(p_i\right) = \sup_M f$
\item $\lim_{i \to \infty} \left|\nabla f\left(p_i\right)\right|_g = 0$
\item $\limsup_{i \to \infty} \Delta f \left(p_i\right) \leq 0$.
\end{enumerate}
\end{lemma}

Pour la preuve, nous renvoyons \`a \cite[Th\'eor\`eme 3.5]{GrahamLee}. Celle-ci reste n\'eanmoins proche de celle du lemme \ref{limsup} qui s'en inspire.

\begin{lemma}\label{pbLin} Soit $\delta \in \bR$. Si $A$ est assez grand, pour tous $g \in \mathcal{C}^{k-2, \alpha}_\delta(M)$ et $h \in \mathcal{C}^{k-1, \alpha}\left(\bint\right)$, il existe une unique solution $u \in \mathcal{C}^{k, \alpha}_\delta$ \`a :

$$
\label{nonlin}
\left\lbrace
\begin{array}{rcl}
- \Delta u + A u & = & g\\
\partial_\nu u + A u & = & h \quad\mathrm{sur~}\bint.
\end{array}
\right.
$$

\end{lemma}

\begin{proof}
Posons $\Omega_i = \rho^{-1}\left(\left[\frac{1}{2^i}; \infty\right[\right)$. Choisissant $i$ assez grand, on peut supposer que $\bint \subset \Omega_i$. $\Omega_i$ est compact donc le probl\`eme :

$$
\left\lbrace
\begin{array}{rcl}
-\Delta u + A u & = & g\quad\mathrm{sur~}\Omega_i\\
\partial_\nu u + A u & = & h\quad\mathrm{sur~}\bint\\
u & = & 0\quad\mathrm{sur~}\partial\Omega_i \setminus \bint
\end{array}
\right.
$$

\noindent admet une unique solution $u_i \in \mathcal{C}^{k, \alpha}\left(\Omega_i\right)$ (On voit facilement que la seule solution du probl\`eme homog\`ene est $0$ et l'op\'erateur est Fredholm d'index 0). En constatant que si $p_0 \in \Omega_i$, $B_{\log 2}(p_0) \subset \Omega_{i+1}$, on a, en utilisant les estimations de Schauder internes \cite[Lemme 4.8]{LeeFredholm} et au niveau de $\bint$ \cite[Lemme 6.29]{GilbargTrudinger} :

\begin{eqnarray*}
\left\| u_i \right\|_{\mathcal{C}^{k, \alpha}_\delta\left(\Omega_{i-1}\right)}
	& \leq & C \left( \| g \|_{\mathcal{C}^{k-2, \alpha}_\delta\left(\Omega_i\right)} + \| h \|_{\mathcal{C}^{k-1, \alpha}\left(\bint\right)}
	 + \| u_i \|_{\mathcal{C}^{0, 0}_\delta\left(\Omega_i\right)} \right)\\
	& \leq & C \left( \| g \|_{\mathcal{C}^{k-2, \alpha}_\delta\left(M\right)} + \| h \|_{\mathcal{C}^{k-1,  \alpha}\left(\bint\right)}
	 + \| u_i \|_{\mathcal{C}^{0, 0}_\delta\left(\Omega_i\right)} \right)
\end{eqnarray*}

\noindent o\`u $C$ est une constante ind\'ependante de $i$. Reste \`a estimer $\| u_i \|_{\mathcal{C}^{0, 0}_\delta\left(\Omega_i\right)}$  :

\begin{eqnarray*}
-\Delta \rho^\delta + A \rho^\delta
	& = & -\rho^2 \left(\overline\Delta \rho^\delta - (n-2) \left\langle\frac{\overline\nabla \rho}{\rho}, \overline\nabla \rho^\delta \right\rangle_{\gbar}\right) + A \rho^\delta\\
	& = & -\rho^2 \left( \delta \rho^{\delta-1} \overline\Delta \rho + \delta(\delta-1) \left|\overline\nabla \rho\right|^2_{\gbar} \rho^{\delta-2}
		-(n-2) \delta \left|\overline\nabla \rho\right|^2_{\gbar}\rho^{\delta-2}\right) + A \rho^\delta\\
	& = & \left(\delta(n-1-\delta) \left|\overline\nabla\rho\right|^2_{\gbar} + A - \delta \rho \overline\Delta\rho\right) \rho^\delta.
\end{eqnarray*}

Donc si $A$ est assez grand, on a $-\Delta \rho^\delta + A \rho^\delta \geq \frac{A}{2} \rho^\delta$. Maintenant :

\begin{eqnarray*}
-\Delta u_i 	& = & -\Delta \left( \rho^\delta \frac{u_i}{\rho^\delta}\right)\\
		& = & -\frac{u_i}{\rho^\delta} \Delta \rho^\delta - \left\langle \nabla \rho^\delta, \nabla \frac{u_i}{\rho^\delta}\right\rangle_g
			- \rho^\delta \Delta \frac{u_i}{\rho^\delta}\\
g - A u_i 	& = & -\frac{u_i}{\rho^\delta} \Delta \rho^\delta - \left\langle \nabla \rho^\delta, \nabla \frac{u_i}{\rho^\delta}\right\rangle_g
			- \rho^\delta \Delta \frac{u_i}{\rho^\delta}\\
\frac{g}{\rho^\delta} - A \frac{u_i}{\rho^\delta}
		& = & -\frac{u_i}{\rho^\delta} \frac{\Delta \rho^\delta}{\rho^\delta}
		      - \left\langle \frac{\nabla \rho^\delta}{\rho^\delta}, \nabla \frac{u_i}{\rho^\delta}\right\rangle_g - \Delta \frac{u_i}{\rho^\delta}\\
\frac{g}{\rho^\delta}
		& = & \frac{u_i}{\rho^\delta} \left(-\frac{\Delta \rho^\delta}{\rho^\delta} + A \right)
		      - \delta \left\langle \frac{\nabla \rho}{\rho}, \nabla \frac{u_i}{\rho^\delta}\right\rangle_g - \Delta \frac{u_i}{\rho^\delta}.
\end{eqnarray*}

En distinguant les cas o\`u $\left|\frac{u_i}{\rho^\delta}\right|$ est maximal en un point int\'erieur \`a $\Omega_i$ et le cas o\`u le maximum est atteint sur $\bint$, on a alors que $\left\| \frac{u_i}{\rho^\delta} \right\|_\infty \leq \frac{2}{A} \| \frac{g}{\rho^\delta} \|_\infty + \frac{1}{B} \|h\|_\infty$ (o\`u on a pos\'e $B=\min_{\bint}\left(\rho^\delta A+\delta \rho^{\delta-1} \partial_\nu \rho\right)$, $B > 0$ si $A$ est assez grand, ce qu'on supposera par la suite). Les fonctions $u_i$ sont donc born\'ees en norme $\mathcal{C}^{2, \alpha}_\delta$ sur les compacts de $M$. Le th\'eor\`eme d'Ascoli et un proc\'ed\'e d'extraction diagonal permettent alors, par une m\'ethode analogue \`a celle utilis\'ee dans la preuve de \cite[Proposition 3.7]{GrahamLee}, de trouver une fonction $u \in \mathcal{C}^{0, 0}_\delta(M) \bigcap \mathcal{C}^{2,0}_{loc}(M)$ qui satisfait \`a :

$$
\left\lbrace
\begin{array}{rcl}
-\Delta u + A u & = & g\quad\mathrm{sur~} M\\
\partial_\nu u + A u & = & h\quad\mathrm{sur~}\bint.\\
\end{array}
\right.
$$

On a alors $u \in \mathcal{C}^{2, \alpha}_\delta(M)$. Par le principe du maximum g\'en\'eralis\'e (lemme \ref{MaxYau}), on constate que $u$ est l'unique solution de ce probl\`eme. Finalement en utilisant la r\'egularit\'e elliptique \cite{LeeFredholm}, on a $u \in \mathcal{C}^{k, \alpha}_\delta(M)$.
\end{proof}

Revenons maintenant \`a la preuve de la m\'ethode de monotonie :

\begin{proof}[D\'emonstration de la m\'ethode de monotonie (Proposition \ref{propMonotonie})]
Posons $\lambda = \min_M \varphi_-$ et $\Lambda = \max_M \varphi_+$. Choissisons une constante $A$ assez grande telle que, pour tout $p \in M$, les fonctions $\varphi \mapsto A \varphi + F(p, \varphi)$ soient croissantes sur $[\lambda; \Lambda]$ de m\^eme que les fonctions $\varphi \mapsto A \varphi + f(p, \varphi)$ pour tout $p \in \bint$ et telle que le lemme \ref{pbLin} soit v\'erifi\'e pour $\delta=0$. On peut alors d\'efinir une suite $\left(\varphi_i\right)_i$ de fonctions telles que $\varphi_0 = \varphi_+$ et telles que $\varphi_{i+1}$ soit l'unique solution dans $\mathcal{C}^{2, \alpha}_0(M)$ de :

$$
\left\lbrace
\begin{array}{rcl}
-\Delta \varphi + A \varphi & = & A \varphi_i + F\left(p, \varphi_i\right)\quad\mathrm{sur~}M\\
\partial_\nu \varphi + A \varphi & = & A \varphi_i + f\left(p, \varphi_i\right)\quad\mathrm{sur~}\bint.\\
\end{array}
\right.
$$

L'existence et l'unicit\'e de $\varphi_{i+1}$ sont garanties par le lemme \ref{pbLin}. La suite des fonctions $\varphi_i$ est d\'ecroissante et minor\'ee par $\varphi_-$. En effet, montrons par exemple qu'on a $\varphi_1 \leq \varphi_0 = \varphi_+$. Ceci revient au m\^eme que montrer que la fonction $\varphi_1 - \varphi_0$ est partout n\'egative. Supposons par l'absurde qu'il existe un point $p \in M$ tel que $\varphi_1(p) - \varphi_0(p) > 0$. On a trois cas :\\

\begin{itemize}
\item Soit le supremum de $\varphi_1 - \varphi_0$ est atteint en un point $p \in \bint$. Dans ce cas $\partial_\nu \left(\varphi_1 - \varphi_0\right) \leq 0$. Mais $\partial_\nu \varphi_1 + A \varphi_1 = f(p, \varphi_0) + A \varphi_0 \leq \partial_\nu \varphi_0 + A \varphi_0$. Ce qui impose $A(\varphi_0-\varphi_1) \geq 0$ : absurde.\\

\item Soit le supremum de $\varphi_1 - \varphi_0$ est atteint en un point int\'erieur \`a $M$ et il existe $0 < \epsilon < \sup_M \left(\varphi_1 - \varphi_0\right)$ et un compact $K \subset \mathring{M}$ tels que sur $M \setminus K$, $\varphi_1 - \varphi_0 \leq \sup_M \left(\varphi_1 - \varphi_0\right) - \epsilon$. Dans ce cas :

$$\int_M \left(-\Delta \psi + A \psi\right) \left(\varphi_1 - \varphi_0\right) \leq 0\qquad\forall~\psi\in \mathcal{C}^{2, 0}_c.$$

Comme $\varphi_0$ et $\varphi_1$ sont localement lipschitziennes et $\psi = 0$ au voisinage de $\partial M$, on a :

$$
\int_M \left( \left\langle\nabla \psi, \nabla \left(\varphi_1 - \varphi_0\right)\right\rangle + A \psi \left(\varphi_1 - \varphi_0\right) \right)\leq 0.
$$

Par densit\'e, ce r\'esultat reste vrai pour toute fonction $\psi$ lipschitzienne \`a support compact. En particulier pour
$\psi_\epsilon = \max \left\{\varphi_1 - \varphi_0 - \frac{\epsilon}{2}, 0 \right\}$ :

$$\int_{\left\{ \varphi_1 \geq \varphi_0 - \frac{\epsilon}{2} \right\}} \left( \left| \nabla \left(\varphi_1 - \varphi_0 - \frac{\epsilon}{2}\right)\right|^2
	+ A \left(\varphi_1 - \varphi_0 - \frac{\epsilon}{2}\right)^2 \right) = \int_M \left( \left| \nabla \psi_\epsilon \right|^2 + A \psi_\epsilon^2 \right) \leq 0$$
	
\noindent absurde car $\psi_\epsilon \neq 0$.\\

\item Sinon, le supremum est atteint \`a l'infini, posons $F = \sup_M \left(\varphi_1 - \varphi_0\right) - \left(\varphi_1 - \varphi_0\right)$. Il existe une suite de points $p_i \in M$ telle que $F(p_i) \to 0$. Par hypoth\`ese, cette suite sort de tout compact de $M$. Quitte \`a extraire une sous-suite, on peut supposer que $p_i \to \widehat{p} \in \binf$. On choisit alors une carte $\left(\rho, \theta^1, \cdots, \theta^{n-1}\right)$ au voisinage de $\widehat{p}$. Posons ensuite $r_i = \frac{\rho(p_i)}{2}$ et

$$g_i(p) = 1 - \frac{\left(\rho(p)-\rho(p_i)\right)^2 + \sum_j \left(\theta^j(p)-\theta^j(p_i)\right)^2}{r_i^2}.$$

On a \mbox{$\max_{\{g_i \geq 0\}} \left|\partial_\alpha g_i\right| \leq \frac{2}{r_i}$},
\mbox{$\max_{\{g_i \geq 0\}} \left|\partial_\alpha \partial_\beta g_i\right| \leq \frac{2}{r_i^2}$} donc
\mbox{$\sup_{\{g_i \geq 0\}} \left|dg_i\right|_g \leq C$}, \mbox{$\sup_{\{g_i \geq 0\}} \left|\Delta g_i\right|_g \leq C$} o\`u $C$ est une constante ind\'ependante de $i$. Soit $q_i$ un point o\`u $\frac{F}{g_i}$ atteint son minimum dans ${\{g_i \geq 0\}}$. $F(q_i) \leq \frac{g_i(q_i)}{g_i(p_i)} F(p_i) \leq F(p_i)$ donc $F(q_i) \to 0$. On d\'efinit ensuite
$h_i = \max\left\{0, \epsilon_i - \frac{F}{g_i}\right\}$ o\`u $\min_{\{g_i \geq 0\}} \frac{F}{g_i} < \epsilon_i < 2\min_{\{g_i \geq 0\}} \frac{F}{g_i}$. On a alors :

\begin{eqnarray*}
A \sup \left(\varphi_1-\varphi_0 \right) \int_M h_i & \leq & \int_M \left(\left\langle \nabla h_i, \nabla F \right\rangle + A F h_i\right)\\
 & \leq & \int_M \left(\left\langle \nabla h_i, \nabla \frac{F}{g_i} \right\rangle g_i + \left\langle \nabla h_i, \nabla g_i \right\rangle \frac{F}{g_i} + A F h_i\right)\\
 & \leq & \int_M \left(-\left|\nabla h_i\right|^2_g g_i-\left\langle \nabla h_i^2, \nabla g_i\right\rangle+\epsilon_i\left\langle \nabla h_i,\nabla g_i\right\rangle\right)\\
 & \leq & \int_M \left(-\left|\nabla h_i\right|^2_g g_i+\left\langle h_i^2, \Delta g_i\right\rangle-\epsilon_i\left\langle h_i,\Delta g_i\right\rangle\right).\\
\end{eqnarray*}

On en d\'eduit en particulier que, pour une certaine constante $C > 0$,

$$0 < \int_M h_i \leq C \epsilon_i \int_M h_i$$

\noindent absurde car $\epsilon_i \to 0$.\\
\end{itemize}

On a donc montr\'e que $\varphi_1 \leq \varphi_0$ partout. On a donc $A \varphi_1 + F(p, \varphi_1) \leq A\varphi_0 + F(p, \varphi_0)$ et $A\varphi_1 + f(p, \varphi_1) \leq A\varphi_0 + f(p, \varphi_0)$. Par r\'ecurrence, on voit que la suite $\varphi_i$ est d\'ecroissante et une preuve analogue \`a la pr\'ec\'edente montre que $\varphi_i \geq \varphi_-$. Montrons maintenant que la suite $\varphi_i$ est born\'ee dans $\mathcal{C}^{2, \alpha}_0$. Pour cela, constatons qu'on a l'in\'egalit\'e suivante dont la preuve est directe :

$$
\left\| f\left(p,\varphi_i\right)\right\|_{C^{1, \alpha}(\bint)}
\leq  C \|f\|_{\mathcal{C}^{1, \alpha}\left(\bint \times [m, M]\right)} \left( 1 + \left\|\varphi_i\right\|_{\mathcal{C}^{1, \alpha}(\bint)}\right)^2
$$

\noindent on en d\'eduit alors l'in\'egalit\'e suivante :

\begin{eqnarray*}
\left\|\varphi_{i+1}\right\|_{\mathcal{C}^{2, \alpha}_0((M)}
	& \leq & C \left( \left\|-\Delta \varphi_{i+1} + A \varphi_{i+1}\right\|_{\mathcal{C}^{0, \alpha}_0(M)}
					 + \left\|\partial_\nu \varphi_{i+1} + A \varphi_{i+1}\right\|_{\mathcal{C}^{1, \alpha}(\bint)}
					 + \left\| \varphi_{i+1} \right\|_{\mathcal{C}^{0, 0}_0(M)} \right)\\
	& \leq & C \left( \left\|F(p, \varphi_i) + A \varphi_i\right\|_{\mathcal{C}^{0, \alpha}_0(M)}
					 + \left\|f(p, \varphi_i) + A \varphi_i\right\|_{\mathcal{C}^{1, \alpha}(\bint)}
					 + \left\| \varphi_{i+1} \right\|_{\mathcal{C}^{0, 0}_0(M)} \right)\\
	& \leq & C' \left( \left\|\varphi_i\right\|_{\mathcal{C}^{0, \alpha}_0(M)}
					 + \left(\left\|\varphi_i\right\|_{\mathcal{C}^{1, \alpha}(\bint)} + 1\right)^2
					 + \left\| \varphi_{i+1} \right\|_{\mathcal{C}^{0, 0}_0(M)} \right)
\end{eqnarray*}

\noindent o\`u pour obtenir la seconde ligne, on a utilis\'e le fait que $F$ est uniform\'ement lipschitzienne en $\varphi$ pour $m\leq\varphi\leq M$. Nous ne pouvons pas utiliser directement l'in\'egalit\'e d'interpolation pour majorer $\left\|\varphi_{i+1}\right\|_{\mathcal{C}^{2, \alpha}_0(M)}$ \`a cause du terme quadratique dans l'in\'egalit\'e pr\'ec\'edente. Il nous faut majorer $\left\|\varphi_{i}\right\|_{\mathcal{C}^{1, \alpha}(\bint)}$. Montrons tout d'abord comment cette majoration permet de conclure. Reprenant l'in\'egalit\'e pr\'ecedente, on en d\'eduit :

\begin{eqnarray*}
\left\|\varphi_{i+1}\right\|_{\mathcal{C}^{2, \alpha}_0}
						    & \leq & C'' \left( \|\varphi_i\|_{\mathcal{C}^{0, \alpha}_0} + 1\right)\\
						    & \leq & C^{(3)}_\mu\left(\|\varphi_i\|_{\mathcal{C}^{0, 0}_0} + 1\right) + \mu \|\varphi_i\|_{\mathcal{C}^{2, \alpha}_0}
						    				 \qquad\textrm{(in\'egalit\'e d'interpolation, \cite[lemme 6.35]{GilbargTrudinger})}\\
						    & \leq & C^{(3)}_\mu\left(\|\varphi_+\|_{\mathcal{C}^{0, 0}_0} + 1\right) + \mu \|\varphi_i\|_{\mathcal{C}^{2, \alpha}_0}.\\
\end{eqnarray*}

Choisissant $\mu \in (0; 1)$, on obtient alors par r\'ecurrence que $\left\|\varphi_i\right\|_{\mathcal{C}^{2, \alpha}_0}$ est born\'e ind\'ependamment de $i$ :

\begin{eqnarray*}
\left\|\varphi_i\right\|_{\mathcal{C}^{2, \alpha}_0}
	& \leq & \frac{C^{(3)}_\mu \left(\|\varphi_+\|_{\mathcal{C}^{0, 0}_0}+1\right)}{1-\mu} \left(1-\mu^i\right)
	+ \mu^i \left\|\varphi_+\right\|_{\mathcal{C}^{2, \alpha}_0}\\
	& \leq & \frac{C^{(3)}_\mu \left(\|\varphi_+\|_{\mathcal{C}^{0, 0}_0}+1\right)}{1-\mu} + \left\|\varphi_+\right\|_{\mathcal{C}^{2, \alpha}_0}.
\end{eqnarray*}

En utilisant les th\'eor\`emes d'Ascoli et de Dini, $\varphi_\infty = \inf_i \varphi_i$ est une fonction $\mathcal{C}^{2, \alpha}_{loc}$ sur $M$ et $\varphi_i$ converge en norme $\mathcal{C}^{2, \alpha}$ sur tout compact de $M$. On a alors que $\partial_\nu \varphi_\infty=f(p, \varphi_\infty)$ et $-\Delta \varphi_\infty = F(p, \varphi_\infty)$. On en d\'eduit que $\varphi_\infty \in \mathcal{C}^{2,\alpha}_0$. En utilisant le lemme \ref{pbLin}, on obtient alors\footnote{C'est ici qu'intervient l'hypoth\`ese $\varphi \geq \epsilon > 0$. En effet $\varphi \mapsto \varphi^{\beta_i}$ n'est pas born\'e en norme $\mathcal{C}^{k, \alpha}$ sur $\bR^*_+$ si $\beta_i < k + \alpha$} $\varphi_\infty \in \mathcal{C}^{k, \alpha}_0$.\\

Montrons finalement comment majorer $\|\varphi_i\|_{\mathcal{C}^{1, \alpha}(\bint)}$. Pour cela, choisissons $\Omega,~\Omega'$ des ouverts r\'eguliers, relativement compacts de M tels que $\bint \subset \overline{\Omega}$ et $\Omega \subset\subset \Omega'$. L'estimation a priori pour les espaces de Sobolev sur $\Omega$ s'\'ecrit :

$$
\left\|\varphi_{i+1} \right\|_{W^{2, p}(\Omega)}
	\leq C \left(\left\|-\Delta \varphi_{i+1} + A \varphi_{i+1}\right\|_{L^p (\Omega)}
		+ \left\|\partial_\nu \varphi_{i+1} + A \varphi_{i+1}\right\|_{W^{1-\frac{1}{p}, p}(\bint)}
		+ \left\|\varphi_{i+1}\right\|_{L^p (\Omega')} \right).
$$

Pour $p = \frac{n}{1-\alpha}$, on a l'injection de Sobolev $W^{2, p}(\Omega) \into \mathcal{C}^{1, \alpha}\left(\overline{\Omega}\right)$ et :

\begin{eqnarray*}
\left\|\varphi_{i+1} \right\|_{\mathcal{C}^{1, \alpha}(\overline{\Omega})}
	& \leq & C' \left(\left\|F(p, \varphi_i) + A \varphi_i\right\|_{L^p (\Omega)}
		+ \left\|f(p, \varphi_i) + A \varphi_i\right\|_{W^{1-\frac{1}{p}, p}(\bint)}
		+ \left\|\varphi_{i+1}\right\|_{L^p (\Omega')} \right)\\
	& \leq & C''\left(\sup_{p\in\Omega}\left|F(p, \varphi_i(p)) + A \varphi_i(p)\right|
		+ \left\|f(p, \varphi_i) + A \varphi_i\right\|_{C^{1, 0}(\bint)}
		+ \sup_{p \in \Omega'}\left|\varphi_{i+1}(p)\right|\right)\\
	& \leq & C''\left(\sup_{p\in\Omega}\left|F(p, \varphi_+(p)) + A \varphi_+(p)\right|
		+ \sup_{p \in \Omega'}\left|\varphi_+(p)\right|\right.\\
	& & \left.+ \|f\|_{\mathcal{C}^{1, 0}\left(\bint \times [m, M]\right)} \left( 1 + \left\|\varphi_i\right\|_{\mathcal{C}^{1, 0}(\bint)}\right)\right)\\
	& \leq & C^{(3)} \left( 1 + \left\|\varphi_i\right\|_{\mathcal{C}^{1, 0}(\overline{\Omega})}\right).
\end{eqnarray*}

En appliquant, comme pr\'ec\'edemment, l'in\'egalit\'e d'interpolation \cite[Lemme 6.35]{GilbargTrudinger}, on obtient :

\begin{eqnarray*}
\left\|\varphi_{i+1} \right\|_{\mathcal{C}^{1, \alpha}(\overline{\Omega})}
	& \leq & \mu \left\|\varphi_i \right\|_{\mathcal{C}^{1, \alpha}(\overline{\Omega})}
		+ C^{(3)}_\mu \left(1+\left\|\varphi_i\right\|_{\mathcal{C}^{0, 0}(\overline{\Omega})}\right)\\
	& \leq & \mu \left\|\varphi_i \right\|_{\mathcal{C}^{1, \alpha}(\overline{\Omega})}
		+ C^{(3)}_\mu \left(1+\left\|\varphi_+\right\|_{\mathcal{C}^{0, 0}(\overline{\Omega})}\right).
\end{eqnarray*}

Un raisonnement identique au pr\'ec\'edent permet de conclure que les $\left\|\varphi_i \right\|_{\mathcal{C}^{1, \alpha}(\overline{\Omega})}$, donc en particulier les $\left\|\varphi_i \right\|_{\mathcal{C}^{1, \alpha}(\bint)}$ sont born\'es.
\end{proof}

\section{R\'esolution de l'\'equation de prescription de la courbure scalaire}

Nous regardons ici le probl\`eme de l'existence et de l'unicit\'e de solutions \`a l'\'equation :

\begin{equation}
-\frac{4 (n-1)}{n-2} \Delta \varphi + \scal~\varphi = \hscal~\varphi^{\kappa+1}
\label{presc_scal}
\end{equation}

\noindent sur une vari\'et\'e $M$ asymptotiquement hyperbolique. $\scal$ d\'esigne ici la courbure scalaire de la m\'etrique $g$ et $\hscal$, la courbure scalaire de $\widehat{g} = \varphi^{\kappa} g$ avec $\kappa = \frac{4}{n-2}$, est une fonction donn\'ee.

\subsection{Probl\`eme de Dirichlet le long de $\bint$}
\subsubsection{Existence et unicit\'e de la solution}

On regarde ici le cas o\`u on prescrit $\varphi=\varphi_0$ le long de $\partial_0 M$ ($\varphi_0 > 0$). On va montrer le th\'eor\`eme suivant :

\begin{theorem}[Solutions de l'\'equation de prescription de la courbure scalaire avec condition au bord de Dirichlet]\label{PrescScalDiri} Soient deux fonctions $\scal,~\hscal \in \mathcal{C}^{k-2, \alpha}_0(M)$ ($k+\alpha \leq l + \beta$) avec $\hscal < 0$ et telles que $\scal,~\hscal \to_{\binf} -n(n-1)$ et $\varphi_0 > 0$ une fonction $\mathcal{C}^{k, \alpha}$ sur $\partial_0 M$ alors il existe une solution unique $\varphi \in~\mathcal{C}^{k, \alpha}_0$ telle que $\varphi > 0$ et $\varphi \to_{\binf} 1$, au probl\`eme de Dirichlet :

\begin{equation}
\label{EqPrescScalDiri}
\left\lbrace
\begin{array}{rcl}
-\frac{4 (n-1)}{n-2} \Delta \varphi + \scal~\varphi & = & \hscal~\varphi^{\kappa+1}\\
\varphi & = & \varphi_0\quad\mathrm{sur}~\bint.\\
\end{array}
\right.
\end{equation}

\end{theorem}

Ce th\'eor\`eme utilise le lemme suivant :

\begin{lemma}\label{limsup} Soit $f : M \to \bR$ une fonction $\mathcal{C}^2$ born\'ee. Il existe une suite $\left(p_i\right)_{i \in \mathbb{N}},~p_i \in M$ telle que :

\begin{enumerate}
\item $f(p_i) \to \underset{\binf}{\mathrm{lim~sup}} f$
\item $\left| \nabla f\left(p_i\right) \right|_g \to 0$
\item $\limsup_{i \to \infty} \Delta f(p_i) \leq 0$.
\end{enumerate}
\end{lemma}

\begin{proof} On choisit une suite de points $q_i \in M_{\frac{1}{i}}$ tels que $\sup_{M_{\frac{1}{i}}}f - f(q_i) \leq \frac{1}{i}$. On a donc $f(q_i) \to \underset{\binf}{\mathrm{lim~sup}}~f$. Quitte \`a extraire une sous-suite, on peut supposer que $q_i$ converge vers un point $\widehat{q} \in \binf$. On pose ensuite $F_i = \sup_{M_{\frac{1}{i}}}(f) + \frac{1}{i} - f$, ainsi $F_i \geq \frac{1}{i}$ sur $M_{\frac{1}{i}}$. Comme pr\'ec\'edemment, on introduit une carte $\left(\rho, \theta^1, \cdots, \theta^{n-1}\right)$ au voisinage de $\widehat{q}$, on pose ensuite $r_i = \frac{\rho(q_i)}{2}$ et

$$g_i(p) = 1 - \frac{\left(\rho(p)-\rho(q_i)\right)^2 + \sum_j \left(\theta^j(p)-\theta^j(q_i)\right)^2}{r_i^2}.$$

On a $\max_{\{g_i \geq 0\}} \left|\partial_\alpha g_i\right| \leq \frac{2}{r_i}$, $\max_{\{g_i \geq 0\}} \left|\partial_\alpha \partial_\beta g_i\right| \leq \frac{2}{r_i^2}$ donc $\sup_{\{g_i \geq 0\}} \left|dg_i\right|_g \leq C$, $\sup_{\{g_i \geq 0\}} \left|\Delta g_i\right|_g \leq C$ o\`u $C$ est une constante ind\'ependante de $i$. On choisit ensuite un point $p_i \in {\{g_i > 0\}}$ tel que $\frac{F_i}{g_i}(p_i)$ soit minimal. Remarquons que $\frac{F_i}{g_i} \to_{g_i\to 0^+} \infty$ donc le minimum est atteint. En un tel point $0 \leq \frac{F_i}{g_i}(p_i) \leq F_i(q_i) \leq \frac{2}{i}$ par cons\'equent on a encore $f(p_i) \to \underset{\binf}{\mathrm{lim~sup}}~f$. On a ensuite :

\begin{eqnarray*}
0 & = & \nabla \log\left(\frac{F_i}{g_i}\right)(p_i)\\
  & = & \frac{\nabla F_i}{F_i}(p_i) - \frac{\nabla g_i}{g_i}(p_i).
\end{eqnarray*}

Donc $\left|\nabla f\right|_g(p_i) = \left|\nabla F_i\right|_g (p_i) = \left|\frac{F_i}{g_i}\right|(p_i) \left|\nabla g_i\right|_g (p_i) \leq \frac{2C}{i} \to 0$. Ensuite, $0 \leq \Delta \log\left(\frac{F_i}{g_i}\right)(p_i) = \frac{\Delta F_i}{F_i}(p_i) - \frac{\Delta g_i}{g_i}(p_i)$. On en d\'eduit que $-\Delta f(p_i) = \Delta F_i \geq \frac{F_i~\Delta g_i}{g_i}(p_i) \to 0$. Ce qui montre que $\underset{i \to \infty}{\mathrm{lim~sup~}} \Delta f(p_i) \leq 0$. 

\end{proof}

\begin{proof}[D\'emonstration du th\'eor\`eme \ref{PrescScalDiri}] La d\'emonstration de ce th\'eor\`eme repose sur la m\'ethode de monotonie (Proposition \ref{propMonotonie}).  L'hypoth\`ese $\hscal < 0$ implique que $\varphi_+ = \Lambda$ avec $\Lambda$ grand est une sur-solution. On peut supposer de plus que $\Lambda \geq \varphi_0$. De m\^eme $\varphi_- = 0$ est une sous-solution naturelle. Cependant, c'est \'egalement une solution et il faut s'assurer que la m\'ethode de monotonie ne converge pas vers cette solution. Nous allons donc modifier $\varphi_-$ au voisinage de $\binf$ pour obtenir une solution $\varphi$ telle que $\varphi \to_{\binf} 1$. Pour cela, on introduit :

\begin{equation}
\label{phiSigma}
\varphi_\sigma = \max \left\{\sigma - \rho, 0\right\}
\end{equation}

Sur $M_\sigma$, si $\sigma$ est assez petit, on a :

\begin{eqnarray*}
-\frac{4(n-1)}{n-2} \Delta \varphi_\sigma + \scal~\varphi_\sigma
	& = & -\frac{4(n-1)}{n-2} \rho^2\left( -\overline{\Delta}\rho + (n-2) \frac{\left|\overline{\nabla}\rho\right|^2_{\gbar}}{\rho}\right) + \scal (\sigma - \rho)\\
	& = & \underbrace{-\frac{4(n-1)}{n-2}\left(-\rho^2\overline{\Delta}\rho+(n-2)\rho\left|\overline{\nabla}\rho\right|^2_{\gbar}\right)}_{\leq 0~\text{si}~\sigma~\text{assez petit}}+\scal (\sigma-\rho)\\
 & \leq & \scal (\sigma-\rho)\\
 & \leq & \hscal (\sigma - \rho)^{\kappa+1}.
\end{eqnarray*}

La derni\`ere in\'egalit\'e provient du fait que $\scal, \hscal \to_{\binf} -n(n-1)$ donc $\mathrm{min}_{M_\sigma} \frac{\scal}{\hscal} \to_{\sigma \to 0} 1$. On peut donc choisir $\sigma > 0$ assez petit tel que $\left(\sigma-\rho\right)^\kappa \leq \sigma^\kappa \leq \mathrm{min}_{M_\sigma} \frac{\scal}{\hscal}$. Ce qui montre que $\varphi_\sigma$ est une sous-solution sur $M_\sigma$ pour $\sigma > 0$ assez petit. Fixons un tel $\sigma > 0$. Quitte \`a diminuer $\sigma$, on peut supposer de plus que $d\rho$ est partout non nul sur $M_{2\sigma}$. Soit $\psi \in \mathcal{C}^2_c(M)$ une fonction test, on a :

\begin{eqnarray*}
\int_M (-\Delta \psi) \varphi_\sigma  & = & \int_M \left\langle \nabla\psi, \nabla\varphi_\sigma\right\rangle_g \qquad\text{($\varphi_\sigma$ est lipschitzienne)}\\
				      & = & \int_{M_\sigma} \left\langle \nabla\psi, \nabla\varphi_\sigma\right\rangle_g +
					\int_{M \setminus M_\sigma} \left\langle \nabla\psi, \nabla\varphi_\sigma\right\rangle_g\\
				      & = & \lim_{\sigma_1 \to 0^+}\left(\int_{M_{\sigma-\sigma_1}} (-\Delta\varphi_\sigma)\psi
				       +\int_{\{\rho=\sigma-\sigma_1\}}\psi\nabla_{N_{\sigma-\sigma_1}}\varphi_\sigma\right.\\
				      &   & \left. + \int_{M\setminus M_{\sigma+\sigma_1}} (-\Delta \varphi_\sigma) \psi
				        - \int_{\{\rho=\sigma+\sigma_1\}}\psi\nabla_{N_{\sigma+\sigma_1}}\varphi_\sigma\right)\\
				      & \leq & -\int_{M_\sigma} (\psi \Delta\varphi_\sigma) - \int_{M \setminus M_\sigma} (\psi \Delta\varphi_\sigma)\\
				      & \leq & \frac{n-2}{4(n-1)}\int_M \left(-\scal~\varphi_\sigma + \hscal~\varphi_\sigma^{\kappa+1}\right) \psi
\end{eqnarray*}

\noindent o\`u on a not\'e, pour $\sigma' \in~(0; 2\sigma)$, $N_{\sigma'}$ la normale unitaire \`a l'hypersurface $\left\{ \rho = \sigma' \right\}$ pointant dans la direction $\rho$ croissant : $d\rho\left(N_{\sigma'}\right) > 0$. Ainsi $\varphi_\sigma$ est une sous-solution au sens des distributions. Il existe donc une solution au probl\`eme de Dirichlet \ref{EqPrescScalDiri}. Il faut maintenant v\'erifier que $\varphi \to 1$ sur $\binf$. Or il existe une suite $p_i \in M$ telle que $\lim_{i\to \infty} \varphi(p_i) = \lim\sup_{\binf} \varphi$, $\lim\sup_{i\to \infty} \Delta\varphi(p_i) \leq 0$. On en d\'eduit, en regardant l'\'equation \eqref{presc_scal} en $p_i$ et en passant \`a la limite :

$$\lim_{i \to \infty} \scal~\varphi(p_i) \leq \lim_{i \to \infty} \hscal~\varphi(p_i)^{\kappa+1}$$

\noindent donc :

$$-n(n-1)\underset{\binf}{\mathrm{lim~sup~}} \varphi \leq -n(n-1)\underset{\binf}{\mathrm{lim~sup~}} \varphi^{\kappa+1},$$

\noindent ainsi $0 \leq \lim\sup_{\binf} \varphi \leq 1$. De la m\^eme fa\c con, on peut trouver une suite $p'_i \in M$ telle que $\lim_{i\to \infty} \varphi(p'_i) = \lim\inf_{\binf} \varphi$, $\lim\inf_{i\to \infty} \Delta\varphi(p'_i) \geq 0$. On trouve alors $\lim\inf_{\binf} \varphi = 0$ ou $\lim\inf_{\binf} \varphi \geq 1$. Cependant le premier cas est impossible car
$\varphi \geq \varphi_\sigma$ et $\lim_{\binf} \varphi_\sigma = \sigma > 0$. Ce qui montre $\varphi \to 1$ au niveau de $\binf$. Par le principe du maximum de Hopf (voir par exemple \cite[Th\'eor\`eme 3.5]{GilbargTrudinger}), $\varphi > 0$. Ceci permet de montrer que $\varphi \geq \epsilon > 0$ pour $\epsilon$ assez petit, ce qui est l'estimation a priori n\'ecessaire pour montrer la r\'egularit\'e de $\varphi$ : $\varphi \in \mathcal{C}^{k, \alpha}_0$. Suivant \cite{DelayCourbScal}, on peut maintenant poser $\varphi = e^\theta$. $\theta$ v\'erifie l'\'equation :

$$-\frac{4(n-1)}{n-2} \left(\Delta \theta + |\nabla \theta |^2_{g} \right) + \scal = \hscal~e^{\kappa \theta}.$$

Supposons maintenant qu'on a deux solutions $\theta_1,~\theta_2$ de cette \'equation telles que $\theta_1 = \theta_2 = 0$ sur $\binf$, $\theta_1 = \theta_2$ sur $\bint$, on a alors en soustrayant :

$$
-\frac{4(n-1)}{n-2} \left(\Delta \left(\theta_1-\theta_2\right) + \left\langle \nabla (\theta_1+\theta_2), \nabla (\theta_1-\theta_2)\right\rangle_g \right) = \hscal \left(e^{\kappa \theta_1}-e^{\kappa \theta_2}\right).
$$

Or :

\begin{eqnarray*}
e^{\kappa \theta_1}-e^{\kappa \theta_2}
	& = & \kappa \int_{\theta_2}^{\theta_1} e^{\kappa \theta} d\theta\\
	& = & \left(\theta_1-\theta_2\right) \kappa \underbrace{\int_{0}^{1} e^{\kappa \theta_x} dx}_{\geq 0}
		    \quad\text{en posant $\theta_x = (1-x) \theta_2 + x \theta_1$}.
\end{eqnarray*}

On conclut alors par le principe du maximum classique que $\theta_2 = \theta_1$. La solution du probl\`eme de Dirichlet est donc unique.
\end{proof}

\subsubsection{Comportement au voisinage de $\binf$}

Nous souhaitons maintenant \'etudier le comportement de la solution $\varphi$ au voisinage de $\binf$. On a le th\'eor\`eme suivant :

\begin{theorem}[Comportement \`a l'infini des solutions de l'\'equation de prescription de la courbure scalaire]\label{comportementBinf} Soient $\delta \in (0, n)$ et $\scal, \hscal \in \mathcal{C}^{k, \alpha}_0$ tels que :
$$
\left\{
\begin{array}{l}
\scal, \hscal \to -n(n-1) \text{ au voisinage de $\binf$}\\
\scal - \hscal \in \mathcal{C}^{k-2, \alpha}_\delta
\end{array}
\right.
$$
alors la solution $\varphi$ de \ref{EqPrescScalDiri} est dans $1 + \mathcal{C}^{k, \alpha}_\delta$.
\end{theorem}

La d\'emonstration de ce th\'eor\`eme repose sur le lemme suivant :

\begin{lemma}\label{compBord} Si $\sigma$ est assez petit,

\begin{itemize}
\item Si $\Lambda \in \bR$, $\Lambda > 1$, il existe une sur-solution $\varphi_+ \in 1 + \mathcal{C}^{k, \alpha}_\delta$ de \eqref{presc_scal} sur
$M_\sigma$ avec $\left.\varphi_+\right|_{S_\sigma} = \Lambda$.
\item Si $0 \leq \lambda < 1$, si on a l'in\'egalit\'e $\frac{4(\delta+1)(n-\delta)}{n-2} > n A_{\kappa+1} (1-\lambda)$ avec
$A_p = \max \left\{p, \frac{p(p-1)}{2}\right\}$, il existe une sous-solution $\varphi_- \in 1 + \mathcal{C}^{k, \alpha}_\delta$ avec $\left.\varphi_+\right|_{S_\sigma} = \lambda$.
\end{itemize}
\end{lemma}

Montrons tout d'abord comment ce lemme implique le th\'eor\`eme. Pour cela posons :

$$
\left\{
\begin{array}{l}
\Lambda_\sigma = \sup_{M_\sigma} \varphi\\
\lambda_\sigma = \inf_{M_\sigma} \varphi.\\
\end{array}
\right.
$$

Comme $\varphi = 1$ sur $\binf$, $\Lambda_\sigma, \lambda_\sigma \to 1$ quand $\sigma \to 0$. On peut donc supposer que $\lambda_\sigma$ v\'erifie l'in\'egalit\'e $\frac{4(\delta+1)(n-\delta)}{n-2} > n A_{\kappa+1} (1-\lambda_\sigma)$. Le lemme fournit alors des fonctions $\varphi_\pm$, d\'efinies sur $M_{\sigma'}$ avec $\sigma' \leq \sigma$, valant respectivement $\Lambda_\sigma$ et $\lambda_\sigma$ sur $M_{\sigma'}$ soient des sur et sous-solutions. En utilisant, comme pour l'unicit\'e de la solution (\`a la fin de la preuve du th\'eor\`eme \ref{PrescScalDiri}), les in\'equations v\'erifi\'ees par $\log \varphi$ et $\log \varphi_\pm$, on obtient $\varphi_- \leq \varphi \leq \varphi_+$ ce qui montre que $\varphi - 1\in \mathcal{C}^{0, 0}_\delta$. On r\'e\'ecrit alors l'\'equation \eqref{presc_scal} sous la forme :

\begin{equation}
-\frac{4(n-1)}{n-2} \Delta \left(\varphi-1\right) + \left(\scal - \hscal(\kappa+1) \int_0^1 (x \varphi + (1-x))^\kappa dx\right) \left(\varphi-1\right) = \hscal-\scal.
\end{equation}

On a alors $\left(\scal - \hscal(\kappa+1) \int_0^1 (x \varphi + (1-x))^\kappa dx\right) \in \mathcal{C}^{k-2,\alpha}_0$. Finalement, par r\'egularit\'e elliptique (voir \cite{LeeFredholm} lemme 4.8), on obtient $\varphi-1 \in \mathcal{C}^{k, \alpha}_\delta$.\\

\begin{proof}[D\'emonstration du lemme \ref{compBord}] Nous allons naturellement chercher des fonctions de la forme :

$$
\left\{
\begin{array}{l}
\varphi_+ = 1 + K \rho^\delta\\
\varphi_- = 1 - k \rho^\delta.
\end{array}
\right.
$$

Pour cela, calculons $\Delta \rho^\delta$ :

\begin{eqnarray*}
\Delta \rho^\delta & = & \rho^2 \left( \overline{\Delta} \rho^\delta - (n-2) \left\langle \frac{\overline{\nabla} \rho}{\rho}, \overline{\nabla} \rho^\delta \right\rangle_{\gbar} \right)\\
	      & = & \delta(\delta-n+1) \rho^\delta |\overline{\nabla}\rho|^2_{\gbar} + \delta \rho^{\delta+1} \overline{\Delta} \rho.
\end{eqnarray*}

$\varphi_+$ est une sur-solution si :

\begin{eqnarray*}
-\frac{4(n-1)}{n-2} \Delta \varphi_+ + \scal~\varphi_+ & \geq & \hscal~\varphi_+^{\kappa+1}\\
-\frac{4(n-1)}{n-2} \Delta u_+ + \scal~u_+ & \geq & \hscal \underbrace{\left(\left(1+u_+\right)^{\kappa+1}-1\right)}_{\text{convexe en $u_+$}} + \hscal-\scal
\end{eqnarray*}

\noindent o\`u on a pos\'e $u_+ = \varphi_+ - 1 = K \rho^\delta$. Ce sera le cas si :

\begin{eqnarray*}
-\frac{4(n-1)}{n-2} \Delta u_+ + \scal~u_+ & \geq & (\kappa+1) \hscal~u_+ + \hscal - \scal\\
-\frac{4(n-1)}{n-2} \Delta u_+ + \left(\scal - (\kappa+1) \hscal\right) u_+ & \geq & \hscal - \scal.\\
\end{eqnarray*}

Or :

\begin{eqnarray*}
&   & -\frac{4(n-1)}{n-2}\Delta u_+ +\left(\scal-(\kappa+1)\hscal\right)u_+\\
& = &\left[-\frac{4(n-1)}{n-2}\left(\delta(\delta-n+1)|\overline{\nabla}\rho|^2_{\gbar}+\delta \rho\overline{\Delta}\rho\right)
	+ \left(\scal - (\kappa+1) \hscal\right) \right] K \rho^\delta\\
& = & \left[-\frac{4(n-1)}{n-2} \delta(\delta-n+1) + \kappa n(n-1) + o(1)\right] K \rho^\delta\\
& = & \left[\frac{4(n-1)}{n-2} (\delta+1)(n-\delta) + o(1)\right] K \rho^\delta.
\end{eqnarray*}

Si on choisit $\sigma$ assez petit, on peut donc supposer que, sur $M_\sigma$ :

$$-\frac{4(n-1)}{n-2} \Delta u_+ + \left(\scal - (\kappa+1) \hscal\right) u_+ \geq \frac{2(n-1)}{n-2} (\delta+1)(n-\delta) K \rho^\delta.$$

Si $K$ est assez grand, on a alors :

$$-\frac{4(n-1)}{n-2} \Delta u_+ + \left(\scal - (\kappa+1) \hscal\right) u_+ \geq \hscal - \scal.$$

On veut ensuite que sur $S_\sigma$, $\varphi_+ = \Lambda$ ce qui conduit \`a choisir $K = \frac{\Lambda - 1}{\sigma^\delta}$ mais, quitte \`a diminuer $\sigma$ ou \`a augmenter $K$, on peut supposer que c'est le cas.\\

Pour trouver une sous-solution, il va nous falloir majorer $(1+u_-)^{\kappa+1}-1$ pour $u_- \in [-1, 0]$ (o\`u comme pr\'ec\'edemment, on a pos\'e $\varphi_- = 1+ u_-$ avec $u_- = - k \rho^\delta$). Admettons pour l'instant l'in\'egalit\'e suivante :

\begin{equation}
\label{inegConv}
(1+u_-)^{\kappa+1}-1-(\kappa+1)u_- \leq A_{\kappa+1} u_-^2.
\end{equation}

On a :

$$\hscal \left[ (1+u_-)^{\kappa+1} - 1 \right] \geq \hscal \left[ (\kappa+1) u_- + A_{\kappa+1} u_-^2\right].$$

$u_-$ \'etant minimal en $\rho = \sigma$, $u_- \geq \lambda-1$ et $u_-^2 \leq (\lambda - 1) u_-$ :

$$\hscal \left[ (1+u_-)^{\kappa+1} - 1 \right] \geq \hscal \left[ (\kappa+1) u_- - A_{\kappa+1} (1-\lambda) u_-\right].$$

Comme pr\'ec\'edemment, on aura une sous-solution si :

$$-\frac{4(n-1)}{n-2} \Delta u_- + \scal~u_-  \leq \hscal \left[ (\kappa+1) u_- - A_{\kappa+1} (1-\lambda) u_-\right] + \hscal - \scal.$$

Or :

\begin{eqnarray*}
&   & -\frac{4(n-1)}{n-2} \Delta u_- + \scal~u_- - \hscal \left[ (\kappa+1) u_- - A_{\kappa+1} (1-\lambda) u_-\right]\\
& = & -\left[\frac{4(n-1)}{n-2}(\delta+1)(n-\delta) - n(n-1) A_{\kappa+1} (1-\lambda) + o(1)\right] k \rho^\delta.
\end{eqnarray*}

Si $\sigma$ est assez petit, on peut supposer que :

$$\hscal - \scal \geq - \epsilon k \rho^\delta$$

pour un certain $\epsilon > 0$. Donc si $k$ est assez grand, $\varphi_-$ est une sous-solution et, comme pr\'ec\'edemment, on peut supposer que $\varphi_- = \sigma$ le long de $S_\sigma$.\\

D\'emontrons maintenant l'in\'egalit\'e \eqref{inegConv}. Pour cela dinstingons deux cas :

\begin{enumerate}
\item Si $\kappa \geq 1$, la fonction $x \mapsto (1+x)^\kappa-1$ est convexe donc elle est toujours au dessus de sa tangente en $x=0$, $(1+x)^\kappa-1 \geq \kappa x$. En
int\'egrant cette in\'egalit\'e entre $u_-$ et $0$, on obtient \eqref{inegConv}.
\item Si $0 \leq \kappa \leq 1$, $\forall~x \in~[-1;0]$ la fonction $\kappa \mapsto (1+x)^{\kappa+1}-1-(\kappa+1) x - (\kappa+1)x^2$ est convexe et n\'egative pour $\kappa = 0$
et $\kappa = 1$. On en d\'eduit donc l'in\'egalit\'e \eqref{inegConv} pour $0 \leq \kappa \leq 1$.
\end{enumerate}
\end{proof}

\subsection{Prescription de la courbure moyenne de $\partial_0 M$}

Sous une transformation conforme, la courbure moyenne se transforme de la mani\`ere suivante\footnote{La courbure extrins\`eque est d\'efinie par rapport \`a la normale $N$ pointant vers l'int\'erieur de M alors que $\nu$ est dirig\'e dans l'autre sens (convention E.D.P.) c'est ce qui explique les signes moins.} :

$$\frac{2}{n-1} \nabla_{\nu} \varphi - H \varphi = - \widehat{H} \varphi^{\frac{\kappa}{2}+1}.$$

On constate que la fonction $\varphi_-$ construite pr\'ec\'edemment (\'equation \ref{phiSigma}), reste une sous-solution pour ce probl\`eme au bord. On souhaite ensuite trouver une sur-solution. Pour cela, constatons qu'on peut toujours supposer qu'on a $H \leq 0$ sur $\partial_0 M$. En effet, il suffit d'effectuer une premi\`ere transformation conforme non triviale uniquement dans un voisinage du bord $\bint$. La fonction $\varphi_+ = \Lambda$ constante est donc une sur-solution si $\Lambda$ est assez grand. On a alors le th\'eor\`eme suivant :

\begin{theorem}[Solutions de l'\'equation de prescription de la courbure scalaire avec courbure moyenne du bord donn\'ee] \label{PrescScalH} Sous les hypoth\`eses du th\'eor\`eme \ref{PrescScalDiri}, si $\widehat{H} \in \mathcal{C}^{k-1, \alpha}\left(\bint\right),~\widehat{H} \geq 0$, il existe une unique solution $\varphi \in \mathcal{C}^{k, \alpha}_0$ telle que $\varphi > 0$ et $\varphi \to_{\binf} 1$ au probl\`eme de la prescription de la courbure scalaire avec courbure moyenne au bord donn\'ee :

\begin{equation}
\label{CourbMoy}
\left\{
\begin{array}{l}
-\frac{4(n-1)}{n-2} \Delta \varphi + \scal~\varphi = \hscal~\varphi^{\kappa+1}\\
\frac{2}{n-1} \nabla_\nu \varphi - H \varphi = - \widehat{H} \varphi^{\frac{\kappa}{2}+1} \text{ le long de }\partial_0 M.
\end{array}
\right.
\end{equation}

De plus si on a, pour un certain $\delta \in (0, n)$, $\scal-\hscal \in \mathcal{C}^{k-2, \alpha}_\delta$, alors
$\varphi \in 1+ \mathcal{C}^{k, \alpha}_\delta$.
\end{theorem}

\begin{proof}La d\'emonstration est similaire \`a celle faite pour le probl\`eme de Dirichlet. Nous n'indiquons donc que les diff\'erences entre les deux preuves. Posons $\gamma = \min \left\{\alpha, \frac{\kappa}{2}\right\}$, la fonction $f(p, \varphi) = H(p) \varphi - \widehat{H} \varphi^{1+\frac{\kappa}{2}}$ est de classe $\mathcal{C}^{1, \gamma}$ sur $\bint \times \bR_+^*$. La m\'ethode de monotonie fournit alors une solution $\varphi \in \mathcal{C}^{1, \gamma}_0(M)$, $\varphi \to_{\binf} 1$. On sait que $\varphi > 0$ sur $\mathring{M}$. Donc $\varphi$ peut \^etre nulle uniquement sur $\bint$. Cependant, en un tel point, on doit avoir $0 = H \varphi - \widehat{H} \varphi^{\frac{\kappa}{2}+1} = \frac{2}{n-2} \nabla_N \varphi < 0$ absurde. Ce qui permet de conclure que $\varphi - 1\in \mathcal{C}^{k, \alpha}_0(M)$. On peut donc comme pr\'ec\'edemment poser $\varphi = e^\theta$. Si $\varphi_1$ et $\varphi_2$ sont deux solutions de \ref{CourbMoy} alors $\theta_1 - \theta_2$ satisfait \`a :

$$
\left\{
\begin{array}{l}
-\frac{4(n-1)}{n-2} \left(\Delta \left(\theta_1-\theta_2\right) + \left\langle \nabla (\theta_1+\theta_2), \nabla (\theta_1-\theta_2)\right\rangle_g \right) = \left(\theta_1-\theta_2\right) \kappa \int_{0}^{1} e^{\kappa \theta_x} dx\\
\frac{2}{n-2} \nabla_\nu \left(\theta_1 - \theta_2\right) + \underbrace{\widehat{H} \frac{n-2}{2} \int_0^1 e^{\frac{\kappa}{2} ((1-x) \theta_1 + x \theta_2)} dx}_{\geq 0} \left(\theta_1 - \theta_2\right) = 0 \text{ le long de $\bint$}\\
\theta_1 - \theta_2 = 0 \text{ le long de $\binf$}.
\end{array}
\right.
$$

On a donc $\theta_1 = \theta_2$ (la d\'emonstration est analogue \`a celle du lemme \ref{pbLin}) et $\varphi$ est unique.
\end{proof}

\section{Construction de solutions des \'equations de contrainte avec des horizons apparents}
\subsection{La m\'ethode Choquet-Bruhat--Lichnerowicz--York}

L'objectif de ce paragraphe est, principalement, de fixer les notations. Nous renvoyons \`a \cite{BartnikIsenberg} pour plus de d\'etails.
Les donn\'ees initiales pour le probl\`eme de Cauchy en relativit\'e g\'en\'erale sont une vari\'et\'e riemanienne $(M, g)$, et un 2-tenseur sym\'etrique $K$ sur $M$ qui s'interpr\`ete comme la seconde forme fondamentale du plongement de $M$ dans $\mathcal{M}$, la solution des \'equations d'Einstein correspondant \`a ces donn\'ees initiales.  Par convention, on choisit ici 
 $$
 K(X,Y)=\langle \nabla_X T,Y\rangle,
 $$
o\`u $T$ est le vecteur unitaire futur, normal \`a $M$. $K$ est reli\'e, dans l'analyse hamiltonienne de la relativit\'e g\'en\'erale, au moment $\pi^{ij}$ conjugu\'e \`a la m\'etrique $g_{ij}$ par :

$$\pi^{ij} = \frac{\partial \mathcal{L}}{\partial\left(\partial g_{ij}\right)} = \sqrt{g}\left(K^{ij}-g^{ij}K\right).$$

Voir par exemple \cite{Carlip}. Ce syst\`eme est contraint et les donn\'ees initiales doivent v\'erifier les \'equations suivantes\footnote{Par la suite, nous ne consid\`ererons que les \'equations d'Einstein du vide.} :

\begin{eqnarray}
\label{hamiltonian} \scal_g - 2 \Lambda_c - \left|K\right|^2_g + \left(\mathrm{tr}_g K\right)^2 = 0 & \qquad\text{(Contrainte hamiltonienne)}\\
\label{momentum}    \mathrm{div}_g K - d\left(\mathrm{tr}_g K\right) = 0 & \qquad\text{(Contraintes moment)}
\end{eqnarray}

\noindent o\`u $\Lambda_c$ est la constante cosmologique et $\left(\mathrm{div}_g K\right)_j = \nabla^i K_{ij}$. D\'ecomposons $K$ en deux parties :
$K = \tau g + L$ o\`u $\tau = \frac{1}{n} \mathrm{tr}_g K$ et $L$ est un tenseur sym\'etrique sans trace. L'\'equation \ref{momentum} devient alors :

$$\mathrm{div}_g L - (n-1) d\tau = 0.$$

Si l'on suppose que $\tau$ est une constante, ce qui signifie qu'on impose \`a notre surface de Cauchy d'\^etre \`a courbure moyenne constante, on obtient donc que $\mathrm{div}_g L=0$. Cette condition est naturelle, voir par exemple \cite{Gerhardt}. $L$ est donc un TT-tenseur. Il existe une bijection simple entre les TT-tenseurs de deux m\'etriques conform\'ement \'equivalentes :

\begin{prop} Si $L$ est un TT-tenseur pour $g$ alors $\hat{L} = \varphi^{-2} L$ est un TT-tenseur pour $\hat{g} = \varphi^\kappa g$ o\`u $\kappa = \frac{4}{n-2}$.
\end{prop}

Ceci sugg\`ere la construction suivante. On choisit $\tau$ une constante, $g$ une m\'etrique sur $M$ et $L$ un TT-tenseur pour $g$. On cherche ensuite une fonction $\varphi > 0$ telle que $\hat{g} = \varphi^\kappa g,~\hat{K} = \varphi^{-2} L + \tau \hat{g}$ soit une solution des \'equations de contraintes. L'\'equation \ref{hamiltonian} donne pour $\varphi$ l'\'equation de Lichnerowicz :

\begin{equation}
\label{eqLichnerowicz}
-\frac{4(n-1)}{n-2} \Delta \varphi + \scal~\varphi - |L|_g^2 \varphi^{-\kappa-3} + \left(n(n-1) \tau^2 - 2 \Lambda_c\right) \varphi^{\kappa+1} = 0.
\end{equation}

Nous allons nous int\'eresser au cas o\`u la constante cosmologique $\Lambda_c$ est n\'egative ou nulle et\linebreak \mbox{$\left(n(n-1) \tau^2 - 2 \Lambda_c\right) > 0$}. Par un changement d'echelle, on voit qu'on peut supposer $n(n-1) \tau^2 - 2 \Lambda_c = n(n-1)$. En particulier $\left|\tau\right| \leq 1$. Ce cas inclut, par exemple, les hypersurfaces asymptotiquement isotropes d'une vari\'et\'e asymptotiquement lorentzienne pour lesquelles $\Lambda_c=0$ et $\tau = \pm 1$ ($\tau = +1$, resp. $-1$, pour les hypersurfaces de Cauchy pour le d\'eveloppement en temps futur, resp. pass\'e) et le cas $\tau=0$, $2 \Lambda_c = -n(n-1)$, des hypersurfaces de Cauchy \`a courbure moyenne nulle dans un espace-temps asymptotiquement anti-de Sitter.

\subsection{La condition au bord d'horizon apparent}\label{secHorizonApparent}

Soit $(M, g)$ une vari\'et\'e asymptotiquement hyperbolique avec un bord interne (\'eventuellement vide). On pose :

$$\bint = \bigcup_i \sigma_i$$

\noindent o\`u les $\sigma_i$ repr\'esentent les composantes connexes de $\bint$. Pour chaque $\sigma_i$, on fixe $\epsilon_i = \pm 1$ qui correspond \`a la condition d'horizon apparent futur ($\epsilon_i = -1$) ou pass\'e ($\epsilon_i = +1$). Sous une transformation conforme, la trace de la seconde forme fondamentale $H_i$ de $\sigma_i$ devient :

\begin{equation}\label{transConfCourbMoy}
\widehat{H}_i = \varphi^{-\frac{\kappa}{2}} \left( H_i + 2 \frac{n-1}{n-2} \frac{\nabla_{N_i} \varphi}{\varphi} \right)
\end{equation}

\noindent o\`u $N_i$ est la normale sortante de l'horizon, i.e. pointant vers l'int\'erieur de $M$. La condition d'horizon apparent pour $\sigma_i$ s'\'ecrit (voir par exemple \cite{Wald} section 12.2):

\begin{eqnarray*}
\widehat{H}_i & = & \epsilon_i \widehat{\tr}_{\sigma_i} \widehat{K}\\
\varphi^{-\frac{\kappa}{2}}\left( H_i+2\frac{n-1}{n-2}\frac{\nabla_{N_i}\varphi}{\varphi}\right) & = & \epsilon_i \left[(n-1)\tau + \varphi^{-2}\widehat{\tr}_{\sigma_i} L\right]\\
\frac{2(n-1)}{n-2}\nabla_{N_i}\varphi + H_i \varphi & = & \epsilon_i \left[(n-1)\tau \varphi^{\frac{\kappa}{2}+1} - L\left(N_i, N_i\right)\varphi^{-1-\frac{\kappa}{2}}\right].
\end{eqnarray*}

Si on pose $\nu_i = - N_i$, on a donc :

\begin{equation}
\label{condBord}
\frac{2(n-1)}{n-2}\nabla_{\nu_i}\varphi - H_i \varphi = \epsilon_i \left[L_{\nu_i\nu_i}\varphi^{-1-\frac{\kappa}{2}} - (n-1)\tau~\varphi^{\frac{\kappa}{2}+1}\right].
\end{equation}

\section{R\'esolution de l'\'equation de Lichnerowicz}\label{secCMC}
Comme pour l'\'etude de la prescription de la courbure scalaire, nous fixons une vari\'et\'e asymptotiquement hyperbolique $(M, g)$ de classe $\mathcal{C}^{l, \beta}$ avec $l + \beta \geq 2$, une constante $\tau \in [-1;~1]$, des constantes $\epsilon_i = \pm 1$ associ\'ees aux composantes connexes du bord int\'erieur et un TT-tenseur $L$ sur $M$.\\

Nous allons montrer l'existence d'une solution \`a l'\'equation de Lichnerowicz \eqref{eqLichnerowicz} dans $\mathcal{C}^{k, \alpha}_\delta$ avec $2 \leq k+\alpha \leq l + \beta$, $0 < \alpha < 1$ et $\delta \in [0; n)$, satifsaisant au niveau du bord interne \`a la condition d'horizon apparent \eqref{condBord}, sous les hypoth\`eses :

\begin{enumerate}
\item $\scal + n(n-1) \in \mathcal{C}^{k-2, \alpha}_\delta$, $\scal + n(n-1) \to_{\binf} 0$,
\item $L \in \mathcal{C}_\frac{\delta}{2}^{k-2, \alpha}\left(M, T^{*2}M\right)$, $L \to_{\binf} 0$,
\item $L$ satisfait \`a $\epsilon_i L_{\nu_i \nu_i} \geq 0$ sur chacune des composantes connexes $\sigma_i$ du bord interne.
\end{enumerate}

Nous devrons de plus supposer que si $\epsilon_i \tau = -1$, le bord $\sigma_i$ a un invariant de Yamabe $\mathcal{Y}(\sigma_i)$ strictement positif. Ceci n'impose de restriction que lorsque $\Lambda_c = 0$. Ce cas appara\^it comme un cas critique pour ce probl\`eme et n'est pas soluble dans le cas g\'en\'eral (voir contre-exemple page \pageref{secCtrEx}). Nous renvoyons au th\'eor\`eme \ref{thmLichneHorizon} page \pageref{thmLichneHorizon} pour un \'enonc\'e plus pr\'ecis du r\'esultat. La construction est bas\'ee sur la m\'ethode de monotonie (Proposition \ref{propMonotonie}).\\

\subsubsection{Construction d'une sous-solution}

Afin de construire une sous-solution, commen\c cons par constater qu'on a une sous-solution naturelle \`a l'int\'erieur de $M$ (en oubliant les conditions au bord interne) donn\'ee par la solution de l'\'equation de Yamabe :

$$-\frac{4(n-1)}{n-2} \Delta \varphi_- + \scal~ \varphi_- + n(n-1) \varphi_-^{\kappa+1} = 0.$$

Montrons le lemme suivant :

\begin{lemma}\label{phiEpsilon} Soient $(M, g)$ une vari\'et\'e asymptotiquement hyperbolique \`a bord de classe au moins $\mathcal{C}^{2, \alpha}$, $0 < \alpha < 1$, $\scal, \hscal \in \mathcal{C}^{0, \alpha}_0$ deux fonctions telles que $\scal = \hscal \to -n(n-1)$ au voisinage de $\binf$ avec $\hscal < 0$. Si $\epsilon >0$ est une constante, on note $\varphi_\epsilon \in \mathcal{C}^{2, \alpha}_0$ la solution du probl\`eme de Dirichlet :

$$
\left\lbrace
\begin{array}{rcl}
-\frac{4 (n-1)}{n-2} \Delta \varphi + \scal~ \varphi & = & \hscal~ \varphi^{\kappa+1}\\
\varphi & = & \epsilon \quad\mathrm{sur}~\partial_0 M\\
\varphi & = & 1 \quad\mathrm{sur}~\binf.
\end{array}
\right.
$$

Alors il existe $\eta > 0$ tel que si $\epsilon > 0$ est assez petit, on a au niveau du bord int\'erieur :

$$\partial_\nu \varphi_\epsilon \leq -\eta.$$
\end{lemma}

\begin{proof} Montrons tout d'abord que si $\epsilon \leq \epsilon'$, $\varphi_\epsilon \leq \varphi_{\epsilon'}$. On a vu (th\'eor\`eme \ref{PrescScalDiri}) que $\varphi_\epsilon, \varphi_{\epsilon'} > 0$. On peut donc poser $\theta_\epsilon = \log \varphi_\epsilon$ (resp. $\theta_{\epsilon'} = \log \varphi_{\epsilon'}$). $\theta_\epsilon, \theta_{\epsilon'}$ v\'erifient alors l'\'equation :

$$-\frac{4(n-1)}{n-2} \left(\Delta \theta_{\epsilon(')} + |\nabla \theta_{\epsilon(')}|^2_g\right) + \scal - \hscal~ e^{\kappa \theta_{\epsilon(')}} = 0.$$

En soustrayant les deux \'equations, on obtient :

$$
 -\frac{4(n-1)}{n-2} \left(\Delta \left(\theta_\epsilon-\theta_{\epsilon'}\right)
+ \left\langle \nabla (\theta_\epsilon+\theta_{\epsilon'}),\nabla (\theta_\epsilon-\theta_{\epsilon'})\right\rangle_g \right)
- \hscal~ \kappa \int_{0}^{1} e^{\kappa \theta_x} dx \left(\theta_\epsilon-\theta_{\epsilon'}\right)=0
$$

\noindent o\`u on a pos\'e $\theta_x = (1-x) \theta_{\epsilon'} + x \theta_\epsilon$. Par le principe du maximum classique, on a alors $\theta_{\epsilon'} \geq \theta_\epsilon$.\\

Les fonctions $\varphi_\epsilon$, $\epsilon \leq \epsilon_0$ sont \'equicontinues sur les compacts de $M$. En effet, on a :

\begin{eqnarray*}
\|\varphi_\epsilon\|_{\mathcal{C}^{2, \alpha}_0(M)}
					& \leq & C \left(\left\| -\Delta\varphi_\epsilon + A\varphi_\epsilon\right\|_{\mathcal{C}^{0, \alpha}_0(M)}
								   + \sup_{M}\left| \varphi_\epsilon\right|\right)\\
					& \leq & C \left(\left\| (A-\scal) \varphi_\epsilon - n(n-1) \varphi_\epsilon^{\kappa+1}\right\|_{\mathcal{C}^{0, \alpha}_0(M)}
									 + \sup_{M}\left| \varphi_\epsilon\right|\right)\\
					& \leq & C'_{\epsilon_0}\left(\left\|\varphi_\epsilon\right\|_{\mathcal{C}^{0, \alpha}_0(M)}
									 + \sup_{M}\left| \varphi_\epsilon\right|\right)\\
					& \leq & C''_{\epsilon_0} (\sup_{M}\left| \varphi_\epsilon\right| + 1) \qquad\textrm{(in\'egalit\'e d'interpolation)}\\
					& \leq & C''_{\epsilon_0} (\sup_{M}\left| \varphi_{\epsilon_0}\right| + 1).
\end{eqnarray*}

Les fonctions $\varphi_\epsilon$ tendent donc uniform\'ement sur tout compact vers une fonction $\varphi_0$ continue. On voit ensuite facilement que $\varphi_0$ est dans $\mathcal{C}^{2,\alpha}_0$ avec $\varphi_0 \neq 0$, que $\varphi_\epsilon \to \varphi_0$ uniform\'ement sur $\Mbar$ (les fonctions $\varphi_\epsilon$ sont toujours plus grandes que la fonction $\varphi_-$ d\'efinie en \eqref{phiSigma} ce qui permet de montrer que $\varphi_0 \to_{\binf} 1$ et $\varphi_\epsilon$ converge simplement en d\'ecroissant vers $\varphi_0$) et que $\varphi_0$ est solution du probl\`eme de Dirichlet :

$$
\left\lbrace
\begin{array}{rcl}
-\frac{4 (n-1)}{n-2} \Delta \varphi + \scal~ \varphi & = & \hscal~ \varphi^{\kappa+1}\\
\varphi & = & 0 \quad\mathrm{sur}~\bint\\
\varphi & = & 1 \quad\mathrm{sur}~\binf.
\end{array}
\right.
$$

On a ensuite :

\begin{eqnarray*}
\| \varphi_\epsilon - \varphi_0 \|_{\mathcal{C}^{2,\alpha}_0}
	& \leq & C \left( \|(-\Delta + A) (\varphi_\epsilon - \varphi_0) \|_{\mathcal{C}^{0,\alpha}_0}
					 + \| \varphi_\epsilon - \varphi_0\|_{\mathcal{C}^{0,0}_0}\right)\\
	& \leq & C'  \|\varphi_\epsilon - \varphi_0 \|_{\mathcal{C}^{0,\alpha}_0}\\
	& \leq & C'' \|\varphi_\epsilon - \varphi_0 \|_{\mathcal{C}^{0,0}_0} \qquad\textrm{(in\'egalit\'e d'interpolation).}
\end{eqnarray*}

Posons $\eta = \frac{1}{2} \min_{\bint} \left| \partial_\nu \varphi_0 \right|$. Le principe du maximum de Hopf \cite{GilbargTrudinger} permet de montrer $\eta > 0$. Si $\epsilon$ est assez petit, l'in\'egalit\'e pr\'ec\'edente montre que \mbox{$\|\varphi_\epsilon - \varphi_0 \|_{\mathcal{C}^{2,\alpha}_0} \leq \frac{1}{2} \min_{\bint} \left| \partial_\nu \varphi_0 \right|$}, on a alors $\partial_\nu \varphi_\epsilon \leq \frac{1}{2} \partial_\nu \varphi_0 \leq -\eta$.
\end{proof}

Pour $\epsilon$ assez petit, on a donc sur $\bint$ :

\begin{eqnarray*}
\frac{2(n-1)}{n-2} \partial_\nu \varphi_\epsilon - H_i \varphi_\epsilon & \leq & -\frac{2(n-1)}{n-2} \eta - H_i \epsilon\\
									& \leq & -\epsilon_i (n-1) \tau  \epsilon^{\frac{\kappa}{2}+1}\\
									& \leq & \epsilon_i \left[L_{\nu_i\nu_i}\epsilon^{-1-\frac{\kappa}{2}}
									  - (n-1)\tau~\epsilon^{\frac{\kappa}{2}+1}\right]\\
									& \leq & \epsilon_i \left[L_{\nu_i\nu_i}\varphi_\epsilon^{-1-\frac{\kappa}{2}}
									  - (n-1)\tau~\varphi_\epsilon^{\frac{\kappa}{2}+1}\right].\\
\end{eqnarray*}

Ainsi $\varphi_\epsilon$ est une sous-solution pour $\epsilon$ assez petit.

\subsubsection{Construction d'une sur-solution}
Dans toute la suite, on pose $\partial_\delta M = \left\lbrace p \in M \vert~d_g(p, \bint) < \delta \right\rbrace$. On choisit $\delta_0$ assez petit tel que l'application :

$$
\begin{array}{cccl}
r : & \partial_{2\delta_0} M & \longrightarrow & \bR \\
    & p			     & \longmapsto     & d_g(p, \partial_0 M)
\end{array}
$$
soit non singuli\`ere. Les hypersurfaces $r=\text{cste}$ sont alors des sous-vari\'et\'es compactes de $M$ naturellement diff\'eomorphes \`a $\partial_0 M$. On note $H_r$ la trace de la courbure extrins\`eque de l'hypersurface $r=\text{cste}$. Rappelons que le laplacien se d\'ecompose en :

\begin{equation}
\label{decompLaplacienGeodesique}
\Delta f = \frac{\partial^2 f}{\partial r^2} + H_r \frac{\partial f}{\partial r} + \Delta_r f,
\end{equation}
o\`u $\Delta_r$ d\'esigne le laplacien associ\'e \`a la m\'etrique induite sur l'hypersurface $r=\text{cste}$.\\

Comme pour la prescription de la courbure scalaire, constatons que si $\Lambda$ est une constante assez grande alors $\varphi_+ = \Lambda$ est une sur-solution de \eqref{eqLichnerowicz}. Nous allons ensuite modifier $\varphi_+$ au voisinage de $\bint$ pour que $\varphi_+$ soit une sur-solution pour la condition au bord d'horizon apparent \eqref{condBord}. Concentrons-nous tout d'abord sur le cas $\epsilon_i \tau > -1$. On remarque tout d'abord que, quitte \`a modifier la m\'etrique par un facteur conforme non trivial au voisinage de $\bint$, on peut supposer que $H_r \geq 0$  et $\scal > 0$ dans un voisinage de $\binf$ (il suffit de choisir une fonction $\varphi = \varphi(r) > 0$ telle que $\frac{2(n-1)}{(n-2)} \varphi'(0) + H \varphi(0)>0$ et $\varphi''(0)<0$ tr\`es grand n\'egativement de sorte \`a avoir $\scal > 0$ sur $\bint$, les fonctions $\scal$ et $H_r$ \'etant continues, elles restent positives dans un voisinage de $\bint$).\\

Soit $f$ la solution de l'\'equation diff\'erentielle suivante :

\begin{equation}
\label{equaDiffLich}
\left\{
\begin{array}{rcl}
f''   & = & \frac{n(n-2)}{4} f^{\kappa+1}\\
f(0)  & = & \Lambda\\
f'(0) & = & 0.
\end{array}
\right.
\end{equation}

$f$ est d\'efinie sur $[0, \delta_\Lambda)$ pour un certain $\delta_\Lambda > 0$. En multipliant par $f'$ et en int\'egrant entre $0$ et $x \in [0, \delta_\Lambda)$
l'\'equation diff\'erentielle, on obtient l'\'equation \og conservation de l'\'energie \fg~ :

\begin{equation}
\label{consNrj}
f'(x)^2 = \frac{(n-2)^2}{4} \left( f(x)^{\kappa+2} - \Lambda^{\kappa+2}\right).
\end{equation}

Cette \'equation du premier ordre s'int\`egre en :

$$ \int_\Lambda^{f(x)} \frac{dy}{\sqrt{y^{\kappa+2} - \Lambda^{\kappa+2}}} = \frac{n-2}{2} x.$$

Finalement en posant $y = \Lambda z$ :

$$ \int_1^{\frac{f(x)}{\Lambda}} \frac{dz}{\sqrt{z^{\kappa+2} - 1}} = \frac{n-2}{2} x \Lambda^{\frac{\kappa}{2}}.$$

L'int\'egrale $I = \int_1^{\infty} \frac{dy}{\sqrt{z^{\kappa+2} - 1}}$ est convergente donc $f(x)$ n'est d\'efini que si $\frac{n-2}{2} x \Lambda^{\frac{\kappa}{2}} < I$, ce qui montre $\delta_\Lambda = \frac{2 I}{n-2} \Lambda^{-\frac{\kappa}{2}}$.\\

Quitte \`a prendre une valeur plus grande pour $\Lambda$, on peut supposer que $\delta_\Lambda < \delta_0$ et $\scal, H_r > 0$ sur $\partial_{\delta_\Lambda} M$. Soit $\delta \in (0;~\delta_\Lambda)$ \`a d\'eterminer plus tard. D\'efinissons :

\begin{equation}
\label{surSolLich}
\varphi_+ = \left\{
\begin{array}{ll}
\Lambda & \text{ sur $M \setminus \partial_\delta M$}\\
f(\delta-r) & \text{ sur $\partial_\delta M$}.\\
\end{array}
\right.
\end{equation}

On a vu que $\varphi_+$ est une sur-solution sur $M \setminus \partial_\delta M$. V\'erifions que $\varphi_+$ est une sur-solution sur $\partial_\delta M$. En utilisant la formule \eqref{decompLaplacienGeodesique}, on a :

\begin{eqnarray*}
	&      & -\frac{4(n-1)}{n-2} \Delta \varphi_+ + \scal~ \varphi_+ - |L|_g^2 \varphi_+^{-\kappa-3}+ n(n-1)\varphi_+^{\kappa+1} \\
	&   =  & -\frac{4(n-1)}{n-2} \left(\partial_r^2 f(\delta-r) + H_r \partial_r f(\delta-r)\right)+\scal~ f(\delta-r)\\
	&		&	 - |L|_g^2 f(\delta-r)^{-\kappa-3}+n(n-1)f(\delta-r)^{\kappa+1}\\
	&   =  & -\frac{4(n-1)}{n-2} \left(f''(\delta-r) - H_r f'(\delta-r)\right)+\scal~ f(\delta-r)\\
	&   &  - |L|_g^2 f(\delta-r)^{-\kappa-3}+n(n-1)f(\delta-r)^{\kappa+1}\\
 & \geq & -\frac{4(n-1)}{n-2} f''(\delta-r) + \scal~ f(\delta-r)-|L|_g^2 f(\delta-r)^{-\kappa-3}+n(n-1)f(\delta-r)^{\kappa+1}\\
 & \geq & \scal~ f(\delta-r)-|L|_g^2 f(\delta-r)^{-\kappa-3}\\
 & \geq & \scal~ \Lambda-|L|_g^2 \Lambda^{-\kappa-3},
\end{eqnarray*}
o\`u, pour obtenir la derni\`ere ligne, on a utilis\'e $\scal > 0$ et $f(x) \geq \Lambda$ pour $x \in [0; \delta_\Lambda)$. Donc, si $\Lambda$ est assez grand, on a que $\varphi_+$ est une sur-solution sur $\partial_\delta M$. Maintenant $\varphi_+$ est de classe $\mathcal{C}^1$ et ses d\'eriv\'ees partielles sont lipschitziennes. Un calcul analogue \`a celui fait dans la preuve du th\'eor\`eme \ref{EqPrescScalDiri} montre alors que $\varphi_+$ est une sur-solution au sens des distributions de \eqref{eqLichnerowicz}. Il reste \`a voir que, pour les conditions au bord \eqref{condBord}, $\varphi_+$ est une sur-solution (pour $\delta$ bien choisi) :

$$\frac{2(n-1)}{n-2}\partial_\nu\varphi_+ -H_i \varphi_+ \geq \epsilon_i \left[ L_{\nu_i\nu_i}\varphi_+^{-1-\frac{\kappa}{2}}-(n-1)\tau\varphi_+^{1+\frac{\kappa}{2}}\right],$$

\noindent c'est-\`a-dire :

$$\frac{2(n-1)}{n-2}f'(\delta) - H_i f(\delta) \geq \epsilon_i \left[ L_{\nu_i\nu_i}f(\delta)^{-1-\frac{\kappa}{2}}-(n-1)\tau f(\delta)^{1+\frac{\kappa}{2}}\right].$$

Or, en utilisant l'\'equation \eqref{consNrj}, ceci revient \`a montrer :

$$(n-1)\sqrt{f(\delta)^{\kappa+2}-\Lambda^{\kappa+2}}-H_i f(\delta)\geq\epsilon_i\left[L_{\nu_i\nu_i}f(\delta)^{-1-\frac{\kappa}{2}}-(n-1)\tau f(\delta)^{1+\frac{\kappa}{2}}\right].$$

Si on choisit $\delta$ proche de $\delta_\Lambda$, on peut obtenir des valeurs aussi grandes qu'on le souhaite pour $f(\delta)$ :

\begin{eqnarray*}
 &   & (n-1)\sqrt{f(\delta)^{\kappa+2}-\Lambda^{\kappa+2}}+(n-1)\epsilon_i \tau f(\delta)^{1+\frac{\kappa}{2}}-H_i f(\delta)-\epsilon_i L_{\nu_i\nu_i}f(\delta)^{-1-\frac{\kappa}{2}}\\
 & = & (n-1) \left(1 + \epsilon_i \tau\right) f(\delta)^{1+\frac{\kappa}{2}} + O\left(f(\delta)\right).
\end{eqnarray*}

Par hypoth\`ese $1 + \epsilon_i \tau > 0$ donc en choisissant $\delta$ assez proche de $\delta_\Lambda$, on peut supposer que :

$$
(n-1)\sqrt{f(\delta)^{\kappa+2}-\Lambda^{\kappa+2}}+(n-1)\epsilon_i \tau f(\delta)^{1+\frac{\kappa}{2}}-H_i f(\delta)
-\epsilon_i L_{\nu_i\nu_i}f(\delta)^{-1-\frac{\kappa}{2}}\geq 0.$$

$\varphi_+$ est alors une sur-solution.\\

Traitons maintenant le cas d'un bord pour lequel $\epsilon_i \tau = -1$. Nous devrons de plus supposer que $\mathcal{Y}(\sigma_i)>0$. Nous allons proc\'eder exactement comme dans le cas $\epsilon_i \tau > -1$, c'est-\`a-dire en modifiant la sur-solution dans un voisinage du bord $\sigma_i$. Nous n'indiquerons donc que les diff\'erences. Remarquons tout d'abord qu'on a le lemme suivant :

\begin{lemma}
Il existe une m\'etrique $g'$ conforme \`a $g$ telle que $g' = g$ en dehors d'un voisinage relativement compact de $\sigma_i$, et telle que 
$$\scal'_{\sigma_i} \text{ est une constante strictement positive et } H'_i(r'=0) = 0, \partial_{r'} H'_i(r'=0) = 0,$$ o\`u $\scal'_{\sigma_i}$ est la courbure scalaire de $\sigma_i$ muni de la m\'etrique induite par $g'$, $r'$ est la fonction distance \`a $\sigma_i$ pour $g'$ et $H'_i$ est la courbure moyenne des hypersurfaces de niveau $\{r' = cste\}$.
\end{lemma}

\begin{proof}
La construction de $g'$ se fait en trois \'etapes. Tout d'abord on construit un premier facteur conforme $u_1$ tel que $u_1 = 1$ sauf dans un voisinage de $\sigma_i$ tel que la m\'etrique $g_1 = u_1^\kappa$ restreinte \`a $\sigma_i$ a une courbure scalaire constante strictement positive. Puis \`a l'aide d'un second facteur conforme  $u_2$ valant $1$ au niveau de $\sigma_i$, $u_2 = 1$ en dehors d'un voisinage de $\sigma_i$ et dont la d\'eriv\'e normale en $\sigma_i$ est bien ajust\'ee, on contruit une m\'etrique $g_2 = u_2^\kappa g_1$ telle que la courbure moyenne de $\sigma_i$ est nulle (voir formule \eqref{transConfCourbMoy}). Finalement, imposer $\partial_{r'} H'(r'=0) = 0$ est plus d\'elicat car sous un changement conforme les hypersurfaces $\{r=cste > 0\}$ changent, ce qui change la d\'efinition de $\partial_r$. L'astuce est de remarquer que si on note $u_3$ le facteur conforme \`a d\'eterminer et qu'on pose $g' = u_3^{-2} g_2$, $r' = d_{g'}(\sigma_i, .)$... on a, en notant $S$ la seconde forme fondamentale du bord $\sigma_i$ (les quantit\'es $r'$, $H'$,... sont associ\'ees \`a la m\'etrique $g'$ alors que $N_2$, $\ric_2$,... sont associ\'ees \`a la m\'etrique $g_2$) :

\begin{eqnarray*}
\partial_{r'} H' + \left| S' \right|^2_{g'}
	& = & - \ric'(N', N') \qquad\text{(\'equation de Mainardi, voir par exemple \cite{Petersen})}\\
	& = & - u_3^2 \ric'(N_2, N_2)\\
	& = & - u_3^2 \left[\ric_{g_2}(N_2, N_2) + (n-2) \frac{\hess_{N_2, N_2}~u_3}{u_3}
			- \left((n-1) \left|\frac{d u_3}{u_3}\right|_{g_2}^2 - \frac{\Delta u_3}{u_3} \right)\right]\\
	& & \text{(transformation conforme du tenseur de Ricci, voir par exemple \cite{Besse})}.
\end{eqnarray*}

Choisissant $u_3$ tel que $u_3(r_2=0) = 1, \partial_{r_2} u_3 (r_2=0) = 0$, on a, en $r = 0$ :
$$\partial_{r'} H' + \left| S \right|^2_g = -\ric_2(N_2, N_2) - (n-1) \partial_{r_2}^2 u_3,$$
ce qui permet de d\'eterminer $\partial_{r_2}^2 u_3$ de mani\`ere \`a avoir $\partial_{r'} H' = 0$ et, comme pour les \'etapes pr\'ec\'edentes, on peut choisir $u_3$ non trivial dans un voisinage de $\sigma_i$.
\end{proof}

Nous allons donc supposer par la suite que $g$ est telle que $\scal_{\sigma_i} = cste > 0$, $H_i (r=0) = 0$ et $\partial_r H_i (r=0) = 0$. R\'e\'ecrivons ensuite l'\'equation \eqref{eqLichnerowicz} sous la forme :

\begin{equation}
\begin{array}{c}
-\frac{4(n-1)}{n-2} \left(\partial_r^2 \varphi + H \partial_r \varphi + \Delta_r \varphi \right) + \left(\scal_r + \left| S \right|^2 - 2 \partial_r H - H^2 \right) \varphi\\
+ n(n-1) \varphi^{\kappa+1} = |L|_g^2 \varphi^{-\kappa-3},
\end{array}
\end{equation}
o\`u $\scal_r$ d\'esigne le scalaire de courbure de (la composante connexe associ\'ee \`a $\sigma_i$ de) l'hypersurface $r = cste$, $S$ sa seconde forme fondamentale et $H$ sa courbure moyenne. Comme pr\'ec\'edemment, soit $h$ la solution de l'\'equation diff\'erentielle suivante :

\begin{equation}
\label{equaDiffLich2}
\left\lbrace
\begin{array}{rcl}
h''   & = & \frac{n(n-2)}{4} h^{\kappa+1} + A h\\
h(0)  & = & \Lambda\\
h'(0) & = & 0,
\end{array}
\right.
\end{equation}
o\`u $A > 0$ est une constante \`a pr\'eciser. L'\'equation de conservation de l'\'energie devient alors :

\begin{equation}
\label{consNrj2}
h'(x)^2 = \frac{(n-2)^2}{4} \left( h(x)^{\kappa+2} - \Lambda^{\kappa+2}\right) + A \left( h(x)^2 - \Lambda^2\right).
\end{equation}

On voit clairement que $h \geq f$ donc $h_\Lambda$ diverge vers $+\infty$ en un temps fini $\mu_\Lambda$ avec $0 < \mu_\Lambda < \delta_\Lambda$. Choisissant $A$ tel que $0 < A < \frac{n-2}{4(n-1)} \scal_{\sigma_i}$, il est alors ais\'e de modifier l'argument pr\'ec\'edent pour en d\'eduire qu'en posant $\varphi_+ = h(\delta - r)$ dans un voisinage $\{ r \leq \delta \}$ de $\sigma_i$ pour $\delta$ assez petit et quitte \`a choisir $\Lambda$ assez grand, on obtient une sur-solution. Seule pose probl\`eme l'estimation du terme $H \partial_r \varphi$. Or, en utilisant l'\'equation \eqref{consNrj2}, on voit qu'il existe une constante $B > 0$ ind\'ependante de $\Lambda > 1$ telle que $h'(r)^2 \leq B h^{\kappa+2}(r)$ pour tout $0 \leq r < \mu_\Lambda$. On en d\'eduit que $$h(r) \leq \left(\frac{2 B}{\kappa} (\mu_\Lambda - r)\right)^{-\frac{2}{\kappa}}.$$ En d\'ecomposant le $h^{\kappa+2}$ qui appara\^it dans l'\'equation \eqref{consNrj2} en $h^{\kappa} h^2$, on obtient : $$h^{\kappa+2}(r) \leq \left(\frac{2 B}{\kappa} (\mu_\Lambda - r)\right)^{-2} h^2(r),$$ ce qui prouve qu'il existe une constante $C > 0$ ne d\'ependant que de $B$ telle que $$h'(r) \leq \frac{C}{\mu_\Lambda - r} h(r).$$ Combin\'e au fait que $H(0) = \partial_r H(0) = 0$, donc que $H(r) = o(r)$, ceci permet de prouver que le terme $H \partial_r \varphi$ est domin\'e par le terme lin\'eaire $\left(\scal_r + \left| S \right|^2 - 2 \partial_r H - H^2 - \frac{4(n-1)}{n-2} A\right) \varphi$ pour $\delta$ assez petit.\\

En utilisant la m\'ethode de monotonie (section \ref{secMonotonie}), on a alors le th\'eor\`eme suivant :

\begin{theorem}[Construction de solutions de l'\'equation de Lichnerowicz contenant des horizons apparents]\label{thmLichneHorizon} ] Soient $M$ une vari\'et\'e asymptotiquement hyperbolique de classe $\mathcal{C}^{l, \beta}$ avec $l+\beta \geq 2$ et $\tau \in [-1; 1]$. Fixons, pour chaque composante connexe $\sigma_i$ de $\bint$, un r\'eel $\epsilon_i = \pm 1$ ($\epsilon_i = -1$ pour un horizon apparent futur et $\epsilon_i = +1$ pour un horizon apparent pass\'e). Supposons de plus, lorsque $\tau = \pm 1$, que, pour tout $i$ tel que $\epsilon_i = -\tau$, l'invariant de Yamabe $\mathcal{Y}(\sigma_i) > 0$. Soit $L \in \mathcal{C}^{k-1, \alpha}_0(M, T^{*2}M)$ un 2-tenseur sym\'etrique de trace nulle tel que $|L|^2_g \to 0$ au voisinage de $\binf$ et tel que $\epsilon_i L_{\nu_i \nu_i} \geq 0$ sur $\bint$. Il existe une solution $\varphi > 0$ au probl\`eme :

\begin{equation}
\label{EqLichHorApp}
\left\lbrace
\begin{aligned}
-\frac{4 (n-1)}{n-2} \Delta \varphi + \scal~ \varphi + n(n-1) \varphi^{\kappa+1} - |L|_g^2 \varphi^{-\kappa-3} = 0 & \quad\text{sur $\mathring{M}$}\\
\frac{2(n-1)}{n-2}\nabla_{\nu_i}\varphi - H_i \varphi = \epsilon_i \left[L_{\nu_i\nu_i}\varphi^{-1-\frac{\alpha}{2}} - (n-1)\tau~\varphi^{\frac{\alpha}{2}+1}\right] & \quad\text{sur $\sigma_i$ pour tout $i$},
\end{aligned}
\right.
\end{equation}
o\`u $\scal$ est le scalaire de courbure de $(M, g)$ et $\nu_i$ la normale \`a $\sigma_i$ sortante (i.e. dirig\'ee vers l'int\'erieur de l'horizon apparent) avec $\varphi \to 1$ au voisinage de $\binf$. De plus si $\scal + n(n-1) \in \mathcal{C}^{k-2, \alpha}_\delta$ et $|L|^2 \in \mathcal{C}^{k-2, \alpha}_\delta$ avec $\delta \in [0; n)$, alors $\varphi-1\in \mathcal{C}^{k, \alpha}_\delta$.
\end{theorem}

\noindent\textbf{Un contre-exemple dans le cas $\epsilon_i \tau = -1$}\label{secCtrEx}
Comme nous l'avons vu, dans le cas limite $\epsilon_i \tau = -1$, la situation est plus complexe. Nous allons montrer que l'existence d'une solution n'est plus garantie. Rappelons que ce cas-ci n'appara\^it que lorsque la constante cosmologique $\Lambda_c$ est nulle. Les difficult\'es apparaissent dans la construction de la sur-solution. Nous allons donc nous concentrer sur le cas $L=0$. Le probl\`eme que nous souhaitons r\'esoudre est alors le suivant :

\begin{equation}
\label{ctrEx}
\left\{
\begin{array}{ll}
-\frac{4(n-1)}{n-2} \Delta \varphi + \scal~\varphi + n(n-1) \varphi^{\kappa+1} = 0 & \text{sur $\mathring{M}$}\\
\frac{2(n-1)}{n-2} \partial_\nu \varphi - H \varphi = (n-1) \varphi^{\frac{\kappa}{2}+1} & \text{sur $\bint$}.
\end{array}
\right.
\end{equation}

Nous allons consid\'erer le cas d'un horizon \`a bord torique :  $M = \mathbb{T}^{n-1} \times ]0; A]$ ($A > 0$) muni de la m\'etrique $g = \frac{1}{y^2} \left( dy^2 + g_0\right)$ o\`u $y$ d\'esigne la composante selon $]0; A]$ et $g_0$ est une m\'etrique sur $\mathbb{T}^{n-1}$ invariante sous les translations. Remarquons que cet espace est la partie du demi-espace de Poincar\'e \og en dessous \fg~de $y=A$ compactifi\'ee par l'action d'un r\'eseau de $\mathbb{R}^{n-1}$. Le groupe d'isom\'etrie contient $\mathbb{T}^{n-1}$. Les fonctions invariantes sous ce groupe de sym\'etrie ne d\'ependent alors que d'une seule coordonn\'ee radiale et l'\'etude des solutions de \eqref{ctrEx} invariantes sous l'action se ram\`ene \`a l'\'etude des solutions d'une \'equation diff\'erentielle. Montrons donc le lemme suivant :

\begin{lemma} Supposons qu'il existe une solution $\varphi \in \mathcal{C}^{2, \alpha}_0$, $\varphi > 0$ au probl\`eme \eqref{ctrEx} telle que $\varphi \to_{\binf} 1$, alors il existe une solution $\tilde{\varphi} \in \mathcal{C}^{2, \alpha}_\delta$ ($\delta \in [0; n[$) \`a \eqref{ctrEx} invariante sous l'action de $\mathbb{T}^{n-1}$ telle que $0<\tilde{\varphi}\leq\varphi$, $\tilde{\varphi} \to_{\binf} 1$.
\end{lemma}

\begin{proof} Reprenons les fonctions $\varphi_\epsilon$ construites dans le lemme \ref{phiEpsilon} dans le cas particulier $\scal = \hscal = -n(n-1)$) :

\begin{equation}
\label{phiEpsilonPb}
\left\lbrace
\begin{array}{rcl}
-\frac{4 (n-1)}{n-2} \Delta \varphi_\epsilon - n(n-1) \varphi_\epsilon & = & -n(n-1) \varphi_\epsilon^{\kappa+1}\\
\varphi_\epsilon & = & \epsilon \quad\mathrm{sur}~\partial_0 M\\
\varphi_\epsilon & \underset{\binf}{\to} & 1.
\end{array}
\right.
\end{equation}

Comme $\varphi > 0$, on peut choisir $\epsilon$ tel que $\varphi > \epsilon$ sur $\bint$. Quitte \`a diminuer $\epsilon$, on peut supposer, de plus, que $\varphi_\epsilon$ est une sous-solution de \eqref{ctrEx}. Comme $\varphi_\epsilon$ est l'unique solution de \eqref{phiEpsilonPb}, elle est invariante sous l'action du groupe d'isom\'etrie. Reprenons la m\'ethode de monotonie (section \ref{secMonotonie}), il est possible de construire la suite des it\'er\'es en partant de $\varphi_\epsilon$. Cette suite est alors croissante et, par r\'ecurrence, la suite des it\'er\'es $\varphi_i$ est invariante sous le groupe des isom\'etries et inf\'erieures \`a $\varphi$. La fonction limite $\tilde{\varphi} = \lim_{i \to \infty} \varphi_i$ est alors une solution $\mathcal{C}^{2, \alpha}_0$ de \eqref{ctrEx} invariante sous le groupe d'isom\'etries, $\varphi_\epsilon \leq \tilde{\varphi} \leq \varphi$. La preuve du th\'eor\`eme \ref{comportementBinf} montre alors que $\tilde{\varphi} \in \mathcal{C}^{2, \alpha}_\delta,~\forall~\delta \in [0; n[$.
\end{proof}

Ce lemme montre que nous pouvons nous contenter de montrer qu'il  n'existe pas de solution $\tilde{\varphi}$ invariante sous le groupe des isom\'etries de $M$. Les tores $y=cste$ ont une courbure moyenne $H = n-1$. Introduisons la coordonn\'ee $r = r(y)$ telle que $r$ soit la distance pour la m\'etrique $g$ au tore $y=A$. Si $\phitil$ est une fonction de $r$, on a (formule \eqref{decompLaplacienGeodesique}) :

$$\Delta \phitil = \partial_r^2 \phitil + H_r \partial_r \phitil = \partial_r^2 \phitil + (n-1) \partial_r \phitil.$$

L'\'equation \eqref{ctrEx} se ram\`ene \`a :

\begin{equation}
\label{ctrEx1}
\left\{
\begin{array}{l}
\phitil''(r)+(n-1)\phitil'(r)+\frac{n(n-2)}{4}\left(\phitil(r)-\phitil^{\kappa+1}(r)\right)=0 \quad\text{sur $[0; \infty)$}\\
- \phitil'(0) = \frac{n-2}{2}\left(\phitil(0) + \phitil^{\frac{\kappa}{2}+1}(0)\right)\\
\phitil(r) \underset{r \to \infty}{\to} 1.
\end{array}
\right.
\end{equation}

Posons $$B(r) = \phitil'(r) + \frac{n-2}{2} \left(\phitil(r) + \phitil^{\frac{\kappa}{2}+1}(r)\right),$$ on a alors :

\begin{eqnarray*}
B'(r) & = & \phitil''(r) + \frac{n-2}{2} \left(\phitil'(r) + \frac{n}{n-2} \phitil^{\frac{\kappa}{2}}(r) \phitil'(r)\right)\\
			& = & -(n-1) \phitil'(r) - \frac{n(n-2)}{4} \left(\phitil(r) - \phitil^{\kappa+1}(r)\right)
					+ \frac{n-2}{2} \phitil'(r) + \frac{n}{2} \phitil^{\frac{\kappa}{2}}(r) \phitil'(r)\\
			& = & \frac{n}{2} \left(\phitil^{\frac{\kappa}{2}}(r) - 1\right) \phitil'(r)
					- \frac{n(n-2)}{4} \phitil(r) \left(1-\phitil^{\frac{\kappa}{2}}\right) \left(1+\phitil^{\frac{\kappa}{2}}\right)\\
			& = & \frac{n}{2} \left(\phitil^{\frac{\kappa}{2}}(r) - 1\right) B(r).
\end{eqnarray*}

Or, par hypoth\`ese, $B(0) = 0$ donc $B(r) = 0$ pour tout $r \geq 0$, ce qui contredit le fait que $\lim_{r \to \infty} B(r) = n-2$ (car $\phitil(r) \to 1$ et $\phitil'(r) \to 0$ \`a l'infini).

\section{Construction de TT-tenseurs}
Par la suite, on fixe une vari\'et\'e asymptotiquement hyperbolique $(M, g)$ contenant \'eventuellement un bord interne $\bint$ et on note $\nu$ la normale sortante de $\bint$ (dans la section \ref{secCMC}, le bord interne correspondra aux horizons apparents de la donn\'ee initiale, $\nu$ sera alors la normale entrante de l'horizon). On fixe un 2-tenseur sym\'etrique de trace nulle $L_0$ (quelconque) puis on cherche une 1-forme $\psi$ telle que $L_{ij} = L_{0ij} + \rlie_{\psi^\sharp} g$ soit un TT-tenseur. On d\'efinit donc le \textbf{laplacien vectoriel} $\Delta_{TT}$ par :

\begin{eqnarray*}
\Delta_{TT} \psi   & = & \text{div} \left( \rlie_{\psi^\sharp} g \right),\\
\Delta_{TT} \psi_j & = & \Delta \psi_j + \nabla^i\nabla_j \psi_i - \frac{2}{n} \nabla_j \left(\nabla^k \psi_k\right).
\end{eqnarray*}

Il est facile de voir que cet op\'erateur est elliptique au sens de \cite{ADN} pour $s_i = 0$ et $t_j = 2$. Une condition au bord naturelle pour $\psi$ est de prescrire $L_{0\nu i}$ sur $\bint$. On introduit donc \'egalement l'op\'erateur $\mathcal{B}$ correspondant aux conditions au bord :

\begin{eqnarray*}
\mathcal{B} \psi   & = & \rlie_{\psi^\sharp} g(\nu, .),\\
\mathcal{B} \psi_i & = & \nabla_\nu \psi_i + \nabla_i \psi_\nu - \frac{2}{n} \nabla^k \psi_k g_{\nu i}.
\end{eqnarray*}

L'op\'erateur $\mathcal{B}$ satisfait la condition de compl\'ementarit\'e au bord\label{condComplBord} (voir \cite{ADN}) pour $r_h = -1$. On pose :
\begin{eqnarray}
\begin{array}{rcccc}
\mathcal{P}^{k, p}_\delta :
W^{k, p}_\delta(M, T^*M) & \to & W^{k-2, p}_\delta(M, T^*M) & \times & W^{k-1-\frac{1}{p}, p}\left(\bint, T^*M\right)\\
			\psi	       & \mapsto & \left(\Delta_{TT} \psi\right. &,       & \left. \mathcal{B} \psi \right)
\end{array}\\
\begin{array}{rcccc}
\mathcal{P}^{k, \alpha}_\delta :
\mathcal{C}^{k, \alpha}_\delta(M, T^*M) & \to& \mathcal{C}^{k-2, \alpha}_\delta(M, T^*M) & \times & \mathcal{C}^{k-1, \alpha}\left(\bint,T^*M\right)\\
			\psi	       & \mapsto & \left(\Delta_{TT} \psi\right. &,       & \left. \mathcal{B} \psi \right).
\end{array}
\end{eqnarray}

Par la suite, nous ne donnerons que les d\'emonstrations dans le cas des espaces de Sobolev, celles dans le cas des espaces de H\"older \'etant analogues. De plus nous abr\'egerons $\mathcal{P}^{k, p}_\delta$ en $\mathcal{P}$.

\subsubsection{Th\'eor\`eme de Fredholm}
\begin{lemma}[Calcul des exposants critiques de $\Delta_{TT}$]\label{lmExpoCrit} Les exposants critiques (au sens de \cite{LeeFredholm}) de $\Delta_{TT}$ sont $s=-2$ et $s=n-1$. Le rayon indicial de $\Delta_{TT}$ est $R = \frac{n+1}{2}$.
\end{lemma}

\begin{proof} Notons $$U^k_{ij} = -\left( \frac{\partial_i \rho}{\rho} \delta_j^k + \frac{\partial_j \rho}{\rho} \delta_i^k - \gbar_{ij} \gbar^{kl}\frac{\partial_l \rho}{\rho} \right).$$ $U^k_{ij}$ satisfait \`a : $$\nabla_i X^j = \nabbar_i X^j + U_{ik}^j X^k,$$
\noindent o\`u $\nabbar$ d\'esigne la connexion de Levi-Civita associ\'ee \`a $\gbar$. D\'ecomposons le calcul en plusieurs \'etapes :

\begin{eqnarray*}
&   & \nabla_i \left( \rho^s \psi_j \right) + \nabla_j \left(\rho^s \psi_i \right) - \frac{2}{n} g_{ij} g^{kl} \nabla_k \left( \rho^s \psi_l \right)\\
& = & \nabbar_i \left( \rho^s \psi_j \right) + \nabbar_j \left(\rho^s \psi_i \right) - \frac{2}{n} g_{ij} g^{kl} \nabbar_k \left( \rho^s \psi_l \right)
  - U^k_{ij} \rho^s \psi_k - U^k_{ji} \rho^s \psi_k + \frac{2}{n} g_{ij} g^{kl} U_{kl}^m \rho^s \psi_m\\
& = & \rho^s \left( \nabbar_i \psi_j + \nabbar_j \psi_i - \frac{2}{n} g_{ij} g^{kl} \nabbar_k \psi_l \right)
  + s \rho^{s-1} \left( \partial_i \rho \psi_j + \partial_j \rho \psi_i - \frac{2}{n} g_{ij} g^{kl} \partial_k \rho \psi_l \right)\\
&&+ 2 \left( \partial_i \rho \delta^k_j + \partial_j \rho \delta^k_i - \gbar_{ij} \gbar^{kl} \partial_l \rho \right) \rho^{s-1} \psi_k
  - \frac{2}{n} \gbar_{ij} \gbar^{kl} \left( \partial_k \rho \delta^m_l + \partial_l \rho \delta^m_k - g_{kl} \gbar^{mp} \partial_p \rho \right) \rho^{s-1} \psi_m\\
& = & \rho^s \left( \nabbar_i \psi_j + \nabbar_j \psi_i - \frac{2}{n} g_{ij} g^{kl} \nabbar_k \psi_l \right)
  + s \rho^{s-1} \left( \partial_i \rho \psi_j + \partial_j \rho \psi_i - \frac{2}{n} g_{ij} g^{kl} \partial_k \rho \psi_l \right)\\
&&+ 2 \left( \partial_i \rho \delta^k_j + \partial_j \rho \delta^k_i - \gbar_{ij} \gbar^{kl} \partial_l \rho \right) \rho^{s-1} \psi_k
  + \frac{2}{n}(n-2) \gbar_{ij} \gbar^{kl} \rho^{s-1} \partial_k \rho \psi_l\\
& = & \rho^s \left( \nabbar_i \psi_j + \nabbar_j \psi_i - \frac{2}{n} g_{ij} g^{kl} \nabbar_k \psi_l \right)
  + (s+2) \rho^{s-1} \left( \partial_i \rho \psi_j + \partial_j \rho \psi_i - \frac{2}{n} g_{ij} g^{kl} \partial_k \rho \psi_l \right).\\
\end{eqnarray*}

Si $T_{ij}$ est un tenseur sym\'etrique de trace nulle :

\begin{eqnarray*}
g^{ik} \nabla_k T_{ij} & = & \rho^2 \gbar^{ik} \left( \nabbar_k T_{ij} - U_{ki}^l T_{lj} - U^l_{kj} T_{il} \right)\\
		       & = & \rho^2 \gbar^{ik} \left[ \nabbar_k T_{ij}
			 +\left(\frac{\partial_k \rho}{\rho}\delta^l_i+\frac{\partial_i \rho}{\rho}\delta^l_k-\gbar_{ki}\gbar^{lm}\frac{\partial_m\rho}{\rho}\right)
			 			T_{lj}\right.\\
		&  & \left. +\left(\frac{\partial_k \rho}{\rho}\delta^l_j+\frac{\partial_i \rho}{\rho} \delta^l_k-\gbar_{kj}\gbar^{lm}\frac{\partial_m\rho}{\rho}\right) T_{il}
			 \right]\\
		       & = & \rho^2 \gbar^{ik} \nabbar T_{ij} - (n-2) \rho \gbar^{ik} \partial_k \rho T_{ij}
\end{eqnarray*}

Finalement :

\begin{eqnarray*}
&   & g^{ik} \nabla_k\left[\nabla_i\left(\rho^s\psi_j\right)+\nabla_j\left(\rho^s\psi_i\right)-\frac{2}{n}g_{ij}g^{kl}\nabla_k\left(\rho^s\psi_l\right)\right]\\
& = & g^{ik} \nabla_k\left[\rho^s \left( \nabbar_i \psi_j + \nabbar_j \psi_i - \frac{2}{n} g_{ij} g^{kl} \nabbar_k\psi_l \right)\right.\\
&   & + \left. (s+2)\rho^{s-1}\left( \partial_i \rho \psi_j + \partial_j \rho \psi_i - \frac{2}{n} g_{ij} g^{kl} \partial_k \rho \psi_l \right)\right]\\
& = & \rho^2 \gbar^{ik} \nabbar_k\left[\rho^s \left( \nabbar_i \psi_j + \nabbar_j \psi_i - \frac{2}{n} g_{ij} g^{kl} \nabbar_k \psi_l \right)\right.\\
&   & \left. + (s+2) \rho^{s-1} \left( \partial_i \rho \psi_j + \partial_j \rho \psi_i - \frac{2}{n} g_{ij} g^{kl} \partial_k \rho \psi_l \right)\right]\\
& & - (n-2)\rho\gbar^{ik}\partial_k \rho \left[\rho^s \left( \nabbar_i \psi_j + \nabbar_j \psi_i - \frac{2}{n} g_{ij} g^{kl} \nabbar_k \psi_l \right)\right.\\
& & + \left.(s+2) \rho^{s-1} \left( \partial_i \rho \psi_j + \partial_j \rho \psi_i - \frac{2}{n} g_{ij} g^{kl} \partial_k \rho \psi_l \right)\right]\\
& = & (s+2)(s-1)\gbar^{ik} \left(\partial_i \rho \psi_j + \partial_j \rho \psi_i - \frac{2}{n} \gbar_{ij} \gbar^{kl} \partial_k \rho \psi_l\right)\\
& & - (s+2)(n-2)\gbar^{ik}\partial_k\rho\left(\partial_i\rho\psi_j+\partial_j\rho\psi_i-\frac{2}{n}\gbar_{ij}\gbar^{kl}\partial_k\rho\psi_l\right) + o(\rho^s)\\
& = & (s+2)(s-n+1)\gbar^{ik}\partial_k\rho\left(\partial_i\rho\psi_j+\partial_j\rho\psi_i-\frac{2}{n}\gbar_{ij}\gbar^{kl}\partial_k\rho\psi_l\right) + o(\rho^s).
\end{eqnarray*}
L'application indiciale \cite[Chapitre 4]{LeeFredholm} de $\Delta_{TT}$ est donc donn\'ee par :
$$I_s\left(\Delta_{TT}\right)(\psi) = (s+2)(s-n+1)\left[\psi+\left(1-\frac{2}{n}\right)\psi\left(\nabbar\rho\right)d\rho\right].$$

Ce qui montre que les exposants critiques de $\Delta_{TT}$ sont $s=-2$ et $s=n-1$.
\end{proof}

Nous allons tout d'abord montrer que $\mathcal{P}$ est semi-Fredholm. Notre d\'emonstration est bas\'ee sur celle de \cite{LeeFredholm}.

\begin{lemma}[Estimation $L^2$ \`a l'infini] Il existe un compact $K$ de $M$ et une constante $C > 0$ tels que si $u \in \mathcal{C}^2_c\left(M\setminus K\right)$, on a : $$\|u\|_{L^2} \leq C \|\Delta_{TT} u\|_{L^2}.$$
\end{lemma}

\begin{proof} En coordonn\'ees, on a :
\begin{eqnarray*}
\Delta_{TT} \psi_k & = & g^{ij}\left( \nabla_i \nabla_j \psi_k + \nabla_i \nabla_k \psi_j - \frac{2}{n} g_{jk} \nabla_i \nabla^l \psi_l \right)\\
		   & = & g^{ij} \nabla_i \nabla_j \psi_k + g^{ij}\left( \nabla_k \nabla_i \psi_j - \riemuddd{l}{j}{i}{k} \psi_l\right)
		 	  - \frac{2}{n} \nabla_k \nabla^l\psi_l\\
		   & = & g^{ij} \nabla_i \nabla_j \psi_k + \ricud{l}{k} \psi_l + \left(1-\frac{2}{n}\right) \nabla_k \left(\nabla^l \psi_l\right).
\end{eqnarray*}

Donc :

\begin{eqnarray*}
\int_M \psi^k \Delta_{TT} \psi_k & = & \int_M \psi^k \left(g^{ij} \nabla_i \nabla_j \psi_k + \ricud{l}{k} \psi_l +
						\left(1-\frac{2}{n}\right) \nabla_k \left(\nabla^l \psi_l\right)\right)\\
				    & = & - \int_M \left[\left(\nabla_i \psi_j \right) \left(\nabla^i \psi^j \right)+ \left(1-\frac{2}{n}\right)
						\left(\nabla^k\psi_k\right)^2 - \ricuu{k}{l}\psi_k \psi_l \right].
\end{eqnarray*}

Or, comme $M$ est asymptotiquement hyperbolique $\ric \simeq -(n-1) g$ au voisinage du bord \`a l'infini. On peut donc choisir $K$ tel que, sur $M\setminus K$, on a $-\ric \geq \frac{n-1}{2} g$. On a alors :
$$\|\psi\|_{L^2} \|\Delta_{TT}\psi\|_{L^2} \geq \left| \int_M \psi^k \Delta_{TT} \psi_k \right| \geq
	\int_M \left(- \ric\right)^{kl}\psi_k \psi_l \geq \frac{n-1}{2} \|\psi\|_{L^2}^2.$$
\end{proof}

\begin{lemma}~
\begin{itemize}
\item Soient $\delta, \delta'$ tels que $\left| \delta + \frac{n-1}{p} - \frac{n-1}{2} \right| < \frac{n+1}{2}$,
$\left| \delta' + \frac{n-1}{p} - \frac{n-1}{2} \right| < \frac{n+1}{2}$ avec $\delta - 1 < \delta' < \delta$. Il existe une constante $C > 0$ telle que, $\forall~\psi \in W^{k, p}_\delta (M, T^*M)$, on a :
$$\| \psi \|_{W^{k, p}_\delta} \leq C \left( \left\| \Delta_{TT} \psi \right \|_{W^{k-2, p}_\delta(M)} + \left\| \mathcal{B} \psi \right\|_{W^{k-1-\frac{1}{p}, p}(\bint)} + \|\psi\|_{L^p_{\delta'}}\right).$$
\item Soient $\delta, \delta'$ tels que $\left| \delta - \frac{n-1}{2} \right| < \frac{n+1}{2}$, $\left| \delta' - \frac{n-1}{2} \right|<\frac{n+1}{2}$ avec $\delta - 1 < \delta' < \delta$. Il existe une constante $C > 0$ telle que, $\forall~\psi \in \mathcal{C}^{k, \alpha}_\delta (M, T^*M)$, on a :
$$\| \psi \|_{\mathcal{C}^{k, \alpha}_\delta} \leq C \left( \left\| \Delta_{TT} \psi \right \|_{\mathcal{C}^{k-2, \alpha}_\delta(M)}
+ \left\| \mathcal{B} \psi \right\|_{\mathcal{C}^{k-1, \alpha}(\bint)} + \|\psi\|_{\mathcal{C}^{0, 0}_{\delta'}}\right).$$
\end{itemize}
\end{lemma}

\begin{proof} On choisit une fonction de troncature lisse $u_\infty$ sur $\Mbar$ \`a valeurs dans $[0; 1]$, telle que $u_\infty = 1$ sur $\binf$ et $u_\infty = 0$ au voisinage de $\bint$. On pose $\psi_\infty = u_\infty \psi$ et $\psi_0 = (1-u_\infty) \psi$. $\psi_\infty$ est nulle au voisinage de $\bint$. Une application du corollaire 6.3 de \cite{LeeFredholm} donne alors :
\begin{eqnarray*}
\|\psi_\infty\|_{W^{k, p}_\delta} & \leq & \left\|\widetilde{Q}\Delta_{TT}\psi_\infty\right\|_{W^{k, p}_\delta} + \left\|\widetilde{T}\psi_\infty\right\|_{W^{k, p}_\delta}\\
				  & \leq & C  \left( \left\|\Delta_{TT}\psi_\infty\right\|_{W^{k-2, p}_\delta} + \left\|\psi_\infty \right\|_{W^{k-1, p}_{\delta'}} \right)\\
				  & \leq & C' \left( \left\|\Delta_{TT}\psi_\infty\right\|_{W^{k-2, p}_\delta} + \left\|\psi_\infty\right\|_{L^p_{\delta'}} \right)\\
\end{eqnarray*}
\noindent o\`u on a utilis\'e l'in\'egalit\'e d'interpolation \cite{And93} :
$$\|\psi_\infty\|_{W^{k-1, p}_{\delta'}} \leq C(\epsilon) \|\psi_\infty\|_{L^p_{\delta'}} + \epsilon \|\psi_\infty\|_{W^{k, p}_{\delta'}}.$$

Les op\'erateurs qui apparaissent ici sont ceux du corollaire 6.3 de \cite{LeeFredholm} :
\begin{eqnarray*}
\widetilde{Q} & : & W^{k-2, p}_\delta \to W^{k, p}_\delta\\
\widetilde{T} & : & W^{k-1, p}_\delta \to W^{k, p}_{\delta'},\\
\end{eqnarray*}
\noindent ils v\'erifient $\widetilde{Q} \Delta_{TT} \psi = \psi + \widetilde{T} \psi$. De m\^eme, comme $\psi_0$ est \`a support compact et que $\mathcal{B}$ satisfait la condition de compl\'ementarit\'e au bord (voir page \pageref{condComplBord}), on peut utiliser les r\'esultats de \cite{ADN} :
$$
\|\psi_0\|_{W^{k, p}} \leq C \left( \left\|\Delta_{TT}\psi_0\right\|_{W^{k-2, p}}+\left\|\mathcal{B}\psi_0\right\|_{W^{k-1-\frac{1}{p}, p}(\bint)}+
\left\| \psi_0\right\|_{L^p}\right).
$$

En additionant et en utilisant \`a nouveau l'in\'egalit\'e d'interpolation :
$$\| \psi \|_{W^{k, p}_\delta} \leq C \left( \left\| \Delta_{TT} \psi \right \|_{W^{k-2, p}_\delta(M)} + \left\| \mathcal{B} \psi \right\|_{W^{k-1-\frac{1}{p}, p}(\bint)} +	\|\psi\|_{L^p_{\delta'}}\right).$$
\end{proof}

\begin{cor}
$\mathcal{P}$ est semi-Fredholm (i.e. $\ker~\mathcal{P}$ est de dimension finie et $\mathrm{Im}~\mathcal{P}$ est ferm\'ee).
\end{cor}

\begin{proof} La d\'emonstration est standard (voir par exemple \cite{Maxwell}). On montre tout d'abord que $\ker~\mathcal{P}$ est de dimension finie. En effet, soit $\psi_i \in \ker~\mathcal{P}$ une suite d'\'el\'ements dans $\ker~\mathcal{P}$. L'injection $W^{k, p}_\delta \into L^p_{\delta'}$ est compacte (car $\delta' < \delta$) on peut donc supposer que la suite converge dans $L^p_{\delta'}$. Cette suite est en particulier de Cauchy, or :
$$\| \psi_i-\psi_j \|_{W^{k, p}_\delta} \leq C \|\psi_i-\psi_j\|_{L^p_{\delta'}}.$$
La suite des $\psi_i$ est donc de Cauchy dans $W^{k, p}_\delta$. Elle converge. Ceci prouve que la boule unit\'e de $\ker~\mathcal{P}$ est compacte et $\ker~\mathcal{P}$ est de dimension finie. $\ker~\mathcal{P}$ admet donc un compl\'ementaire ferm\'e $F$ dans
$W^{k, p}_\delta$. Montrons maintenant qu'il existe une constante $\widetilde{C}$ telle que $\forall \psi \in F$ :
$$\|\psi\|_{W^{k, p}_\delta}\leq\widetilde{C}\left(\left\|\Delta_{TT}\psi\right\|_{W^{k-2, p}_\delta(M)}+\left\|\mathcal{B}\psi\right\|_{W^{k-1-\frac{1}{p}, p}(\bint)}\right).$$
Dans le cas contraire, on peut trouver une suite $\psi_i \in F$ telle que $\|\psi_i\|_{W^{k-2, p}_\delta(M)} = 1$ et
$\left\|\Delta_{TT}\psi\right\|_{W^{k-2, p}_\delta(M)}+\left\|\mathcal{B}\psi\right\|_{W^{k-1-\frac{1}{p}, p}(\bint)} \leq \frac{1}{i}$. Quitte \`a extraire une sous-suite, on peut supposer que la suite des $\psi_i$ converge dans $L^p_{\delta'}$. On a alors :
$$\|\psi_i-\psi_j\|_{W^{k, p}_\delta} \leq C \left( \frac{1}{i} + \frac{1}{j} + \|\psi_i - \psi_j\|_{L^p_{\delta'}}\right).$$

La suite des $\psi_i$ est donc de Cauchy dans $W^{k, p}_\delta$. Elle converge vers un \'el\'ement de norme 1 tel que $\mathcal{P} \psi = 0$. Absurde. On en d\'eduit finalement que $\mathrm{Im}~\mathcal{P}$ est ferm\'ee (en effet, $\mathcal{P} : F \to \mathrm{Im}~\mathcal{P}$ est un isomorphisme bicontinu).
\end{proof}

Avant de prouver que $\mathcal{P}$ est un isomorphisme, nous allons tout d'abord montrer que :

\begin{lemma}\label{RegEllip2} Soit $\psi \in L^{p_0}_{\delta_0}$ avec $1 < p_0 < \infty$ et $\left| \delta_0 + \frac{n-1}{p_0} - \frac{n-1}{2} \right| < \frac{n+1}{2}$.
\begin{enumerate}
\item Si :
$$
\left\lbrace
\begin{array}{l}
\Delta_{TT} \psi \in W^{k-2, p}_{\delta}\\
\mathcal{B} \psi \in W^{k-1-\frac{1}{p}, p}
\end{array}
\right.
$$
\noindent avec $2 \leq k \leq l + \beta$, $1 < p < \infty$ et $\left| \delta + \frac{n-1}{p} - \frac{n-1}{2} \right| < \frac{n+1}{2}$,
Alors $\psi \in W^{k, p}_\delta$.\\
\item Si :
$$
\left\lbrace
\begin{array}{l}
\Delta_{TT} \psi \in \mathcal{C}^{k-2, \alpha}_{\delta}\\
\mathcal{B} \psi \in \mathcal{C}^{k-1, \alpha}
\end{array}
\right.
$$
\noindent avec $0 < \alpha < 1$, $2 \leq k + \alpha \leq l+\beta$ et $\left| \delta - \frac{n-1}{2} \right| < \frac{n+1}{2}$,
Alors $\psi \in \mathcal{C}^{k, \alpha}_\delta$.
\end{enumerate}
\end{lemma}

\begin{proof} Comme pr\'ec\'edement, on d\'ecompose $\psi = \psi_0 + \psi_\infty$. $\psi_\infty$ est nulle dans un voisinage de $\bint$ donc on peut appliquer la proposition 6.5 de \cite{LeeFredholm} (Remarquons que $\Delta_{TT}\psi_\infty = \Delta_{TT} \psi$ au voisinage de $\binf$) : $\psi_\infty \in W^{k, p}_\delta$. De m\^eme pour $\psi_0$, on peut appliquer le th\'eor\`eme 10.5 de \cite{ADN} : $\psi_0 \in W^{k, p}_c \subset W^{k, p}_\delta$. On en d\'eduit $\psi \in W^{k, p}_\delta$.
\end{proof}

Ceci montre que si $\psi \in \ker~\mathcal{P}$ est dans un certain $L^p_{\delta_0}$ avec $\left| \delta_0 + \frac{n-1}{p_0} - \frac{n-1}{2} \right| < \frac{n+1}{2}$,  alors $\psi \in W^{2, 2}_0$. De plus si $\psi \in \mathcal{C}^{0, 0}_{\delta_0}$, avec $\left| \delta_0 - \frac{n-1}{2} \right| < \frac{n+1}{2}$, on voit que $\psi \in L^p_{\delta_1}$  pour $\delta_1$ l\'eg\`erement plus petit que $\delta_0$ et $p$ grand donc en appliquant ce qui pr\'ec\`ede $\psi \in W^{2, 2}_0$. Ce qui permet de conclure que si $\psi \in \ker~\mathcal{P}$, $\psi \in W^{2, 2}_0$. Montrons finalement le r\'esultat suivant :

\begin{lemma} Soit $M$ une vari\'et\'e asymptotiquement hyperbolique. On suppose qu'il n'existe aucun champ de vecteur de Killing conforme $L^2$ sur $M$. Alors $\mathcal{P}$ est un isomorphisme.
\end{lemma}

\begin{proof} Proc\'edons comme \cite{Maxwell}. Nous savons tout d'abord que le noyau de $\mathcal{P}^{k, l}_\delta$ est \'egal \`a celui de $\mathcal{P}^{2, 2}_0$. Or si $\psi \in \ker~\mathcal{P}^{2, 2}_0$, on a :
\begin{eqnarray*}
0 & = & \int_M \left\langle \psi, \Delta_{TT} \psi \right\rangle\\
  & = & \int_{\bint} \left\langle \psi, \mathcal{B} \psi \right\rangle
      - \int_M \left| \rlie_{\psi^\sharp} g \right|^2_g\\
  & = & - \int_M \left| \rlie_{\psi^\sharp} g \right|^2_g
\end{eqnarray*}
$\psi^\sharp$ est alors un champ de vecteurs de Killing conforme $L^2$. Ceci montre que $\mathcal{P}^{k, p}_\delta$ est injectif. Pour montrer qu'il est surjectif, il suffit de montrer que $\mathcal{P}^{2, 2}_0$ est surjectif. En effet, $\mathcal{C}^{k, 0}_c(M, T^*M) \times \mathcal{C}^{k, 0}(\bint, T^*M)$ est dense dans $W^{k-2, p}_\delta(M, T^*M) \times W^{k-1-\frac{1}{p}, p}\left(\bint, T^*M\right)$. S'il existe une solution $\psi \in W^{2, 2}_0$ au probl\`eme $\Delta_{TT} \psi = u$, $\mathcal{B} \psi = b$ avec $u \in \mathcal{C}^{k, 0}_c(M, T^*M)$ et $b \in \mathcal{C}^{k, 0}(\bint, T^*M)$ alors le lemme \ref{RegEllip2} permet alors de conclure que $\psi \in W^{k, p}_\delta$. Ceci montre alors que $\im \mathcal{P}^{k, p}_\delta \supset \mathcal{C}^{k, 0}_c(M, T^*M) \times \mathcal{C}^{k, 0}(\bint, T^*M)$ est dense dans $W^{k-2, p}_\delta(M, T^*M) \times W^{k-1-\frac{1}{p}, p}\left(\bint, T^*M\right)$.\\

Montrons donc que $\mathcal{P}^{2, 2}_0 : W^{2, 2}_0 \to L^2(M) \times W^{\frac{1}{2}, 2}(\bint)$ est surjectif. Pour cela, il suffit de montrer que son adjoint est injectif. Or le dual de $L^2(M) \times W^{\frac{1}{2}, 2}(\bint)$ est $L^2(M) \times H^{-\frac{1}{2}}(\bint)$. Soit
$(\widetilde{\psi}, b) \in \ker \left(\mathcal{P}^{2,2}_0\right)^*$. On a alors \cite{Hormander3} que $\widetilde{\psi}$ et $b$ sont de classe $\mathcal{C}^{l, \beta}$ avec $\widetilde{\psi} \in W^{2, 2}_0$. Si $\psi$ est \`a support compact sur $\mathring{M}$ :
$$\int_M \left\langle \widetilde{\psi}, \Delta_{TT} \psi \right\rangle_g = 0.$$

L'op\'erateur $\Delta_{TT}$ \'etant formellement autoadjoint, ceci prouve que $\Delta_{TT} \widetilde{\psi} = 0$. Si maintenant $\psi \in W^{2, 2}_0$ est quelconque :
\begin{eqnarray*}
0 & = & \int_{\bint} \left\langle \mathcal{B} \psi, b \right\rangle_g + \int_M \left\langle \widetilde{\psi}, \Delta_{TT} \psi \right\rangle_g\\
  & = & \int_{\bint} \left\langle \mathcal{B} \psi, b + \widetilde{\psi} \right\rangle_g + \int_{\bint} \left\langle \mathcal{B} \widetilde{\psi}, \psi \right\rangle_g.
\end{eqnarray*}

Or il est facile de voir qu'on peut toujours construire des 1-formes $\psi$ r\'eguli\`eres telles que $\psi|_{\bint}$ et $\mathcal{B}\psi$ soient des fonctions (r\'eguli\`eres) donn\'ees. On en d\'eduit donc que $b+\widetilde{\psi} = 0$ et $\mathcal{B}\widetilde{\psi} = 0$. On a ensuite :
\begin{eqnarray*}
0 & = & \int_{\bint} \left\langle \mathcal{B} \widetilde{\psi}, l \right\rangle_g + \int_M \left\langle \widetilde{\psi}, \Delta_{TT} \widetilde{\psi} \right\rangle_g\\
  & = & - \int_M \left| \rlie_{\psi^\sharp} g\right|^2_g.
\end{eqnarray*}

$\widetilde{\psi}$ est un champ de vecteurs de Killing conforme $L^2$. Ce qui montre $\widetilde{\psi} = 0$ et, comme $b+\widetilde{\psi} = 0$, $b=0$ : $\mathcal{P}^{2, 2}_0$ est surjectif.
\end{proof}

\subsubsection{Champs de vecteurs de Killing conformes}
\begin{lemma}Soit $X$ un champ de vecteurs de Killing conforme sur $M$. Si $X \in L^2(M, g)$ alors $X=0$.\end{lemma}

\begin{proof} On a tout d'abord $X \in W^{2, 2}_0(M)$. En effet $X$ v\'erifie l'\'equation elliptique $\Delta_{TT} X^\flat = 0$. Un processus standard d'it\'eration montre alors que $X \in \mathcal{C}^{0, \alpha}_0(M)$ , $0 < \alpha < 1$. On en d\'eduit (en utilisant le lemme 3.7 de \cite{LeeFredholm}) que $X \in \mathcal{C}^{0, \alpha}_{\alpha-1} \hookrightarrow \mathcal{C}^{0, \alpha}_{(0)}$. On voit qu'on doit avoir $X = 0$ sur $\binf$ car sinon $X \not\in L^2(M)$. $X$ est ensuite un champ de vecteurs de Killing conforme sur $(\Mbar, \gbar)$. Par la suite nous identifierons le champ $X$ et sa forme duale associ\'ee \`a la m\'etrique $\gbar$, ainsi $X_j=\gbar_{ij}X^i$. En particulier, on a 
$$\overline{\Delta}_{TT} X = 0.$$
Donc $X \in \mathcal{C}^{l, \beta}_{(0)}$. Les d\'eriv\'ees dans des directions tangentes \`a $\binf$ de $X$ sont nulles. Ceci implique que $\nabbar X = 0$ sur $\binf$. En effet, soit $\hat{p} \in \binf$, choisissons une carte $\phi = \left(\rho, \theta^{1}, \cdots, \theta^{n-1}\right)$ au voisinage de $\hat{p}$. Par la suite, on notera $\theta^{0} = \rho$. Quitte \`a red\'efinir les $\theta^i$ pour $i>0$, on peut supposer $\gbar_{ij}(\hat{p}) = \delta_{ij}~(i, j \geq 0)$. Les d\'eriv\'ees $\partial_i X_j(\hat{p})$ pour  $i > 0$, $j \geq 0$ correspondent \`a des d\'erivations de $X$ dans des directions tangentes \`a $\bint$ et sont nulles. De plus comme $X(\hat{p}) = 0$, $\nabla_i X_j = \partial_i X_j$ ($i, j \geq 0$). l'\'equation satisfaite par $X$ en $\hat{p}$ s'\'ecrit :
$$\partial_i X_j + \partial_j X_i - \frac{2}{n} \partial^k X_k \delta_{ij} = 0$$

En particulier, appliqu\'ee \`a $i=0$ et $j > 0$, l'\'equation pr\'ec\'edente devient :
\begin{eqnarray*}
0 & = & \partial_0 X_j + \partial_j X_0\\
  & = & \partial_0 X_j \qquad\textrm{(car $X_0$ est nul le long de $\binf$)}
\end{eqnarray*}

puis \`a $i=j=0$ :
\begin{eqnarray*}
0 & = & 2 \partial_0 X_0 - \frac{2}{n} \partial^k X_k\\
  & = & \left(2 - \frac{2}{n}\right) \partial_0 X_0 - \frac{2}{n} \sum_{k > 0} \partial^k X_k\\
  & = & \left(2 - \frac{2}{n}\right) \partial_0 X_0.
\end{eqnarray*}

Ce qui prouve que toutes les d\'eriv\'ees de $X$ sont nulles en $\hat{p}$.\\

On \'elargit ensuite $(\Mbar, \gbar)$ au niveau de $\binf$ pour avoir une vari\'et\'e $(\widetilde{M}, \widetilde{g})$ $\mathcal{C}^{k, \beta}$ et on prolonge $X$ par $0$ \`a tout $\widetilde{M}$. $\widetilde{X}$ est alors un champ de vecteur de Killing conforme $\mathcal{C}^{1, 1}$, nul sur le collier qu'on a ajout\'e \`a $\Mbar$. Comme $\widetilde{\Delta}_{TT} X = 0$ au sens des distributions, $\widetilde{X} \in \mathcal{C}^{k, \beta}(\widetilde{M})$. On conclut alors, en utilisant le th\'eor\`eme d'unique continuation des champs de vecteurs de Killing conformes \cite[Th\'eor\`eme 6.4]{Maxwell} valable pour les m\'etriques $W^{2, p}$ avec $p > n$, que $\widetilde{X} = 0$.
\end{proof}

On en d\'eduit donc le th\'eor\`eme suivant :

\begin{theorem}\label{isoTT} Soit $(M, g)$ une vari\'et\'e asymptotiquement hyperbolique de classe $\mathcal{C}^{l, \beta}$ avec $l+\beta \geq 2$. Les op\'erateurs
\begin{eqnarray*}
\begin{array}{rcccc}
\mathcal{P}^{k, p}_\delta :
W^{k, p}_\delta(M, T^*M) & \to & W^{k-2, p}_\delta(M, T^*M) & \times & W^{k-1-\frac{1}{p}, p}\left(\bint, T^*M\right)\\
			\psi	       & \mapsto & \left(\Delta_{TT} \psi\right. &,       & \left. \mathcal{B} \psi \right),
\end{array}\\
\begin{array}{rcccc}
\mathcal{P}^{k, \alpha}_\delta :
\mathcal{C}^{k, \alpha}_\delta(M, T^*M) & \to& \mathcal{C}^{k-2, \alpha}_\delta(M, T^*M) & \times & \mathcal{C}^{k-1, \alpha}\left(\bint,T^*M\right)\\
			\psi	       & \mapsto & \left(\Delta_{TT} \psi\right. &,       & \left. \mathcal{B} \psi \right)
\end{array}
\end{eqnarray*}
sont des isomorphismes.
\end{theorem}

En particulier, appliquant ce th\'eor\`eme \`a l'\'equation
$$\left\lbrace
\begin{array}{l}
0 = \divg L = \Delta_{TT} \psi + \divg L_0\\
b = L(\nu, .) = \mathcal{B} \psi + L_0(\nu, .),
\end{array}
\right. $$
on obtient le th\'eor\`eme suivant de construction des TT-tenseurs :

\begin{theorem}[Construction de TT-tenseurs sur une vari\'et\'e asymptotiquement hyperbolique]\label{thmConstrTT} Soit $(M, g)$ une vari\'et\'e asymptotiquement hyperbolique de classe $\mathcal{C}^{l, \beta}$ avec $l+\beta \geq 2$.
\begin{enumerate}
\item Soient $L_0$ un 2-tenseur sym\'etrique et sans trace avec $L_0 \in W^{k-1, p}_\delta$ avec $2 \leq k \leq l$, $1 < p < \infty$ et $\left|\delta + \frac{n-1}{p} - \frac{n-1}{2} \right| < \frac{n+1}{2}$, et $b \in W^{k-1-\frac{1}{p}, p} (\bint, T^*M)$. Il existe une unique 1-forme $\psi \in W^{k, p}_\delta(M, T^*M)$ telle que, si on d\'efinit $L = L_0 + \rlie_{\psi^\sharp} g$, alors :
\begin{itemize}
\item $L$ est un tenseur sym\'etrique transverse et sans trace
\item $L(\nu, .) = b$ sur $\bint$.
\end{itemize}
\item Soient $L_0$ un 2-tenseur sym\'etrique et sans trace avec $L_0 \in \mathcal{C}^{k-1, \alpha}_\delta$ avec $2 \leq k + \alpha \leq l$,
$0 < \alpha < 1$ et $\left|\delta - \frac{n-1}{2} \right| < \frac{n+1}{2}$, et $b \in \mathcal{C}^{k-1, \alpha} (\bint, T^*M)$. Il existe une unique 1-forme $\psi \in \mathcal{C}^{k, \alpha}_\delta(M, T^*M)$ telle que, si on d\'efinit
$L = L_0 + \rlie_{\psi^\sharp} g$, alors :
\begin{itemize}
\item $L$ est un tenseur sym\'etrique transverse et sans trace
\item $L(\nu, .) = b$ sur $\bint$.
\end{itemize}
\end{enumerate}
\end{theorem}

\bibliographystyle{amsalpha}
\bibliography{biblio}

\end{document}